\documentclass[12pt]{amsart}
\usepackage{amssymb}

\usepackage{fullpage}
\usepackage{graphicx}
\usepackage{epsfig}
\usepackage{subcaption}
\usepackage{fullpage,amsmath,amsfonts,epstopdf}

\theoremstyle{plain}
\newtheorem{theorem}{Theorem}[section]
\newtheorem{lemma}[theorem]{Lemma}
\newtheorem{proposition}[theorem]{Proposition}
\newtheorem{corollary}[theorem]{Corollary}
\theoremstyle{definition}

\newtheorem{example}[theorem]{Example}
\newtheorem{conjecture}[theorem]{Conjecture}

\numberwithin{equation}{section}

\newcommand{\R}{\mathbb{R}}
\newcommand{\N}{\mathbb{N}}
\newcommand{\Z}{\mathbb{Z}}

\newcommand{\C}{\mathbb{C}}

\renewcommand{\phi}{\varphi}

\renewcommand{\Im}{\mathrm{Im}}

\title[Embeddings of interval exchange transformations \\ 
into planar piecewise isometries]{Embeddings of interval exchange transformations \\ 
into  planar piecewise isometries}
\author{ Peter Ashwin, Arek Goetz, Pedro Peres and Ana Rodrigues}
\address{PP, PA and AR: Department of Mathematics, University of Exeter, Exeter EX4 4QF, UK}
\address{AG: Department of Mathematics, San Francisco State University, San Francisco, USA}

\date{\today}

\begin{document}

\maketitle

\begin{abstract}
Although \emph{piecewise isometries} (PWIs) are higher dimensional generalizations of one dimensional \emph{interval exchange transformations} (IETs), their generic dynamical properties seem to be quite different. In this paper we consider embeddings of IET dynamics into PWI with a view to better understanding their similarities and differences. We derive some necessary conditions for existence of such embeddings using combinatorial, topological and measure theoretic properties of IETs. In particular, we prove that continuous embeddings of minimal 2-IETs into orientation preserving PWIs are necessarily trivial and that any $3$-PWI has at most one non-trivially continuously embedded minimal $3$-IET with the same underlying permutation. Finally, we introduce a family of $4$-PWIs with apparent abundance of invariant nonsmooth fractal curves supporting IETs, that limit to a trivial embedding of an IET.
\end{abstract}

\section{Introduction}

{\em Interval exchange transformations} (IET) are bijective piecewise translations of an interval divided into a finite partition of subintervals. \emph{Piecewise isometries} (PWIs) \cite{Goetz1,Goetz2} are generalizations of IETs to higher dimension where a region is split into a number of convex sets (usually polytopes) and these are rearranged using isometries. Both IETs and PWIs arise in a number of applications. For example, PWIs in two dimension have been found in models used for signal processing and digital filters \cite{ACP97,Da,Dean06,KWC}, for Hamiltonian systems \cite{S, SHM}, for printing processes \cite{AKMTW} or for other types of geometric dynamics \cite{Schwartz,VV}. PWIs exhibit complex and diverse dynamical behaviour that is far less understood than, and quite different from, that of IETs. There are many results that suggest generic choices of parameters for IETs give ergodicity while many examples suggest that this is rarely the case for PWIs in dimension two or more. In this paper, we discuss the general problem of embedding IET dynamics within PWIs with a particular focus on the regularity of this embedding for two dimensional PWIs. 

IETs were defined by Keane \cite{Ke} and studied for instance in \cite{AOW85,Bos85,Gal85}.  Masur and Veech \cite{M,V1} established unique ergodicity of  typical IETs while Avila and Forni \cite{AF} showed that a typical IET is either weakly mixing or an irrational rotation. It is known that IETs (and suspension flows over IETs with roof function of bounded variation) are not strongly mixing \cite{CFS,KA}.

We define an IET as in \cite{AF} (see also \cite{CFS,Ke}). Let $d \geq 2$ be a natural number and let $\pi$ be an irreducible permutation of $\{1,...,d\}$, that is, such that $\pi(\{1,...,k\})\neq \{1,...,k\}$ for $1\leq k < d$. Let $\mu \in \mathbb{R}_+^d$ and define
\begin{equation}\label{eq0}
x_0=0, \quad x_j=\sum_{k=1}^{j}\mu_k, \quad 1\leq j \leq d.
\end{equation}
We consider an interval $I=\left[x_0, x_d \right)$ partitioned into subintervals $I_j=[x_{j-1},x_j)$ for $1\leq j \leq d$. An {\em interval exchange transformation} (more precisely, a $d$-IET on the interval $I$) is a pair $(I,f_{\mu,\pi})$ where $f_{\mu,\pi}:I\rightarrow I$ is the bijection that rearranges $I_j$ according to $\pi$. For $x \in I_j$ we write $f_{\mu,\pi}=f_j$ where
\begin{equation}\label{eq3}
f_j(x)=x+ \tau_j,
\end{equation}
and $\tau_j=\sum_{\pi(k)<\pi(j)}\mu_k - \sum_{k<j}\mu_k$.

We define a two dimensional, orientation-preserving PWI as follows (see \cite{Goetz1}). Let $r\geq 2$ be a natural number, let $X$ be a subset of $\mathbb{R}^2$ (which we parametrize as $\C$) and let $P=\{X_0,\ldots,X_{r-1}\}$ for $r>1$ be a finite partition of $X$ into convex sets (or {\em atoms}), that is, $\bigcup_{0 \leq i <r} X_i =X$, and $X_i \cap X_j = \emptyset$ for $i \neq j$. We say $(X,T)$ is a {\em piecewise isometry} (more precisely, an orientation preserving $d$-PWI in two dimensions) if $T$ is such that
for $z \in X_j$ we have $T(z)=T_j(z)$ with
$$
T_j(z)=e^{i\theta_j}z+\lambda_j,
$$
for some $\theta_j\in[0,2\pi)$ and $\lambda_j\in\C$, so that $T$ is a piecewise isometric rotation or translation.  There will be a subset of points (maximal invariant set) that remain in $X$ for all forward iterates under $T$. Potentially this set could have dimension less than $2$. Note that the restriction of the atoms to the maximal invariant set need not be convex.

Many examples of PWIs have been studied in recent years; for example \cite{BLP} study a class of piecewise rotations on the square and computed numerically box-counting dimensions, correlation dimensions and complexity of the symbolic language produced by the system. Adler, Kitchens and Tresser \cite{AKT} investigated a specific class of nonergodic piecewise affine maps of the torus and give a decomposition into three invariant sets whose dynamics is very different. They show that the map on one of these invariant set is minimal, uniquely ergodic and an odometer; they also demonstrate a the existence of a full Lebesgue measure set of periodic points. It was proved by Buzzi \cite{Buz} that piecewise isometries have zero topological entropy. Lowenstein and  Vivaldi \cite{LV14} present a computer-assisted proof for renormalizing a one-parameter family of piecewise isometries of a rhombus.

In general, for a given PWI it is helpful to define a partition of $X$ into a {\it regular} and an {\it exceptional set} \cite{AG06a}. If we consider the zero measure set given by the union $\mathcal E$ of all preimages of the set of discontinuities $D$, then we say its closure $\overline{\mathcal E}$ (which may be of positive measure) is called the {\it exceptional set} for the map. The complement of the exceptional set is called the {\it regular set} for the map and consists of disjoint polygons or disks that are periodically coded by their itinerary through the atoms of the PWI. As an example, in \cite{AKT} the authors show for a particular transformation where the rotations are rational, the regular set has full Lebesgue measure and as a consequence, the exceptional set has zero Lebesgue measure. However as highlighted in \cite{AFu} there is numerical evidence that the exceptional set may have positive Lebesgue measure for typical PWIs: certainly this is the case for transformations that are products of mixing IETs.

Even when the exceptional set has positive Lebesgue measure, as noted in \cite{AG06a} there is numerical evidence that Lebesgue measure on the exceptional set may not be ergodic - there can be invariant curves that prevent trajectories from spreading across the whole of the exceptional set.  In \cite{AG06}, a planar PWI whose generating map is a permutation of four cones was investigated, and coexistence of an infinite number of periodic components and of an uncountable number of transitive components was proved.
On these transitive components it was noted that the dynamics is conjugate to a transitive interval exchange. In  \cite{Ash1,AG06a},  similar maps were examined and the existence of a large number of these invariant curves, apparently nowhere smooth, are investigated.


\begin{figure}[t]
	\begin{subfigure}{.49\textwidth}
		\centering
		\includegraphics[width=1\linewidth]{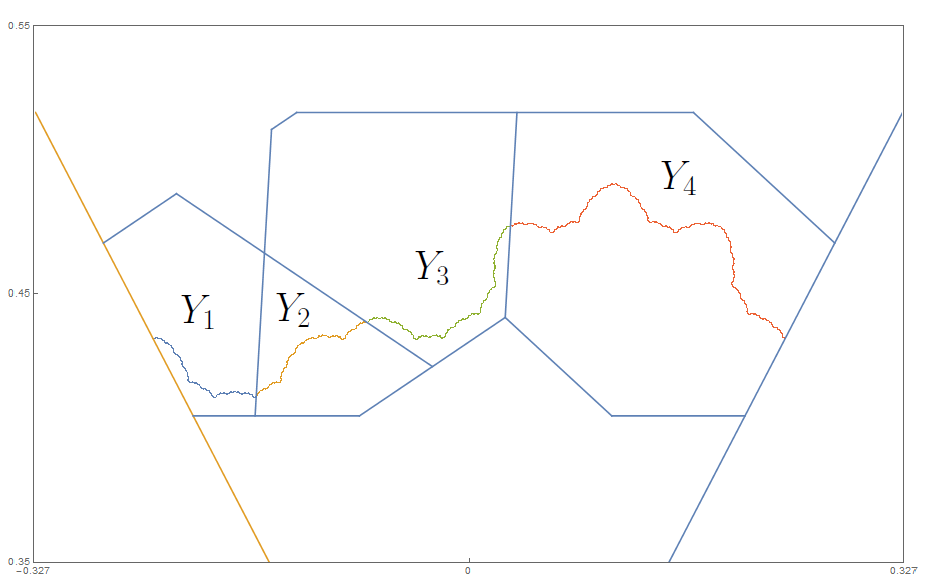}
		\caption{}
		\label{fig:partition}
	\end{subfigure}
	\begin{subfigure}{.49\textwidth}
		\centering
		\includegraphics[width=1\linewidth]{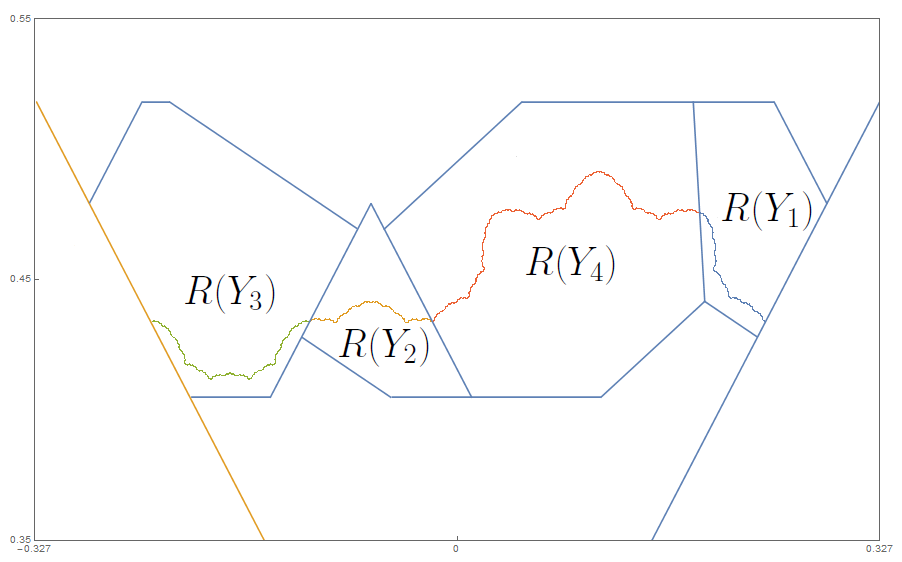}
		\caption{}
		\label{fig:partition1}
	\end{subfigure}
	\centering
	\captionsetup{width=1\textwidth}
	\caption{An illustration of the action of a piecewise isometry $R$ (see \eqref{eq40}), on the image of a non-trivial embedding $Y=\bigcup_{j=1}^{4}Y_j$ of a minimal $4$-IET. (A) An invariant set $Y$ where each $Y_j$ for $j=1,...,4$,  is contained in a polygon. Points in each polygon are mapped isometrically by $R$ to a subset of the region $\{ z\in \mathbb{C}: 0.35< \Im(z)<0.55 \}$. (B) Image of $Y$ and the polygons in (A) under $R$.}
	\label{fig:invariantcurve}
\end{figure}

In this paper we consider general properties of an embedding of an IET into a PWI, and consider conditions for this embedding to be trivial or non-trivial. Our main results are as follows. 
\begin{itemize}
\item In Theorem \ref{thm:necessary} we use combinatorial properties of IETs to prove that in order for a PWI realize a continuous embedding of an IET with the same permutation, its parameters must satisfy a necessary condition: the \emph{parametric connecting equation} \eqref{ce22}.
\item As a consequence of this, Theorem \ref{thm:trivial2IET}, states that all continuous embeddings of minimal 2-IETs are trivial and Theorem \ref{thm:3PWI} asserts that a $3$-PWI has at most one non-trivially continuously embedded minimal $3$-IET with the same underlying permutation.
\item Given an IET embedded into a PWI we, use the derived \emph{tangent exchange map} \eqref{eq9} to prove Theorem \ref{t:ergodiccondition}, which gives a necessary condition on the parameters of a PWI such that there is a continuous embedding of an IET into that PWI. 
\end{itemize}


We introduce a specific example $T$ (\ref{e:pwi4}) of a PWI that has a trivially embedded IET on the boundary. Considering $R$, a first return map under $T$ to a subset of the phase space we observe invariant regions bounded by invariant curves (Figure \ref{fig:invstrip}) and perform numerical experiments to verify the conditions of Theorems \ref{thm:necessary} and \ref{t:ergodiccondition}. We introduce a PWI $T'$ (see \eqref{Tprime}) on 3 atoms that apparently exhibits a single invariant curve that is a non-trivial embedding of a $3$-IET into $T'$.
Using this we make specific conjectures about the nature of non-trivial embeddings of IETs in PWIs.

This paper is organized as follows.  In Section~\ref{sec:embeddings} we consider possible embeddings of a transitive IET into a PWI, and make some definitions regarding their regularity. We identify trivial cases of embedding as where the image of the embedding is either a union of lines or of arcs of the same radius. Furthermore, we extend the Rauzy-Veech induction for IETs to PWIs that admit continuous embeddings of IETs. In Section~\ref{sec:connecting} we introduce some combinatorial conditions on the embedding of an IET into a PWI and state a necessary condition for existence of continuous embeddings. 
Using these techical tools, we prove  that only trivial embeddings of $2$-IETs are possible and that a $3$-PWI has at most one non-trivially continuous embedded $3$-IET with the same underlying permutation.
In Section~\ref{sec:ergodic} we turn to ergodic properties of the embeddings and in Theorem~\ref{t:ergodiccondition} give a necessary condition for embedding in terms of average returns. In Section~\ref{sec:example} we introduce concrete examples of PWIs and show numerical results. We introduce a PWI on 3 atoms, illustrate some examples of orbits for this piecewise isometry and numerically estimate the parameters of a $3$-IET which is embedded into this PWI. We also introduce a particular planar $4$-PWI illustrated in Figure~\ref{fig_iet_cones} that is an ``IET with a twist''. This transformation has a trivially embedded $2$-IET on a line that we call the \emph{baseline} and arbitrarily close to this baseline there are non-trivial rotations. The dynamics of points close to this baseline is remarkably rich. In particular, numerical simulations suggest that the baseline is an accumulation for non-smooth invariant curves that are non-trivial embeddings of $4$-IETs in the $4$-PWI.  We illustrate some examples of orbits for this piecewise isometry and show numerical evidence for abundance of periodic orbits for certain regions of the parameters. We show that the parameters of this map satisfy the restrictions from Theorem~\ref{thm:necessary}. We numerically verify that the condition from Theorem~\ref{t:ergodiccondition} is satisfied. 
 Section~\ref{sec:discussion} is a discussion that considers some open questions and possible generalisations of these results.

\section{Symbolic, topological and differentiable embeddings}
\label{sec:embeddings}

In this section we introduce some definitions of various regularity properties that characterize an embedding of an IET into a PWI. The weakest of these is a symbolic embedding. Furthermore, we extend  Rauzy-Veech induction for IETs to PWIs that admit continuous embeddings of IETs.

Consider a $d$-IET $(I,f_{\mu,\pi})$ which we sometimes denote by $(I,f)$ when parameters are clear from context. For a  point $x \in I$ we define the {\em itinerary} or {\em symbolic encoding} of $x$ by the IET as
\begin{equation}\label{eq6}
i(x)=i_0 i_1\ldots \in \Sigma(d),
\end{equation}
where $\Sigma(d)$ represents one-sided sequences of elements and $i_k\in\{1,\cdots,d\}$ is such that $f^k(x) \in I_j$ if and only if $i_k=j$. 

Similarly, suppose that $(X,T)$ is a $d$-PWI with atoms $\{X_j\}_{j=1}^d$. We define the itinerary of a point $z \in X$ by the PWI as
\begin{equation}
i'(z)=i'_0 i'_1... \subset \Sigma(d)
\end{equation}
where $i'_k\in\{1,\cdots,d\}$ is such that $T^k(x) \in X_j$ if and only if $i'_k=j$.

We now introduce some definitions that will be used throughout this paper.

An injective map $h:I\rightarrow X$ is a {\em symbolic embedding} of $(I,f)$ into $(X,T)$ if $h(I)\subset X$ is an invariant set for $(X,T)$ and there is a numbering of the atoms such that 
 \begin{equation*}
 i'\circ h(x)=i(x)~~\mbox{ for all }x\in I.
 \end{equation*}

An injective map $h:I\rightarrow X$ is a {\em piecewise continuous embedding} of $(I,f)$ into $(X,T)$ if $h(I)\subset X$ is an invariant set for $(X,T)$, $h|_{I_j}$ is a homeomorphism for each $j$, and $h(I_j)\subset X_j$. In this case note that
\begin{equation}\label{eq0s2}
h\circ f(x) = T\circ h(x)~~\mbox{ for all }x\in I.
\end{equation}

If $(I,f)$ has a piecewise continuous embedding $h$ into $(X,T)$ then it is also a symbolic embedding, but the converse does not necessarily hold (to see this, note that $h(I)$ need not be closed if it is a disconnected union of disjoint orbits). If $h$ is a piecewise continuous embedding that is continuous on $I$, we say it is a {\em continuous embedding}, otherwise we say it is a \emph{discontinuous} embedding.

We say $h$ is a {\em differentiable embedding} if it is a piecewise continuous embedding and $h|_{I_j}$ is continuously differentiable. 

We characterize certain differentiable embeddings as, in some sense, trivial. A piecewise continuous embedding of $(I,f)$ by $h$ into $(X,T)$ is a {\em linear embedding} if there are $z_j,v_j\in\C$ such that
\begin{equation}\label{eq1s2}
h|_{I_j}(x)=z_j+v_j x,
\end{equation}
for all $x \in I$, and is an {\em arc embedding} if there are $z_j\in\C$, $r_j>0$ and $a_j,b_j\in \R$ such that
\begin{equation}\label{eq2s2}
h|_{I_j}(x)=z_j+r_j\exp[i (a_j x+b_j)],
\end{equation}
for all $x \in I$. We say an embedding is {\em trivial} if it is a linear embedding or an arc embedding, otherwise it is {\em non-trivial}.


\begin{lemma}\label{l1s2}
For any $d$-IET $(I,f)$ there exists a trivial continuous embedding $h:I\rightarrow X$ of $(I,f)$ into a $d$-PWI $(X,T)$, which can be either a linear embedding or an arc embedding. Suppose in addition that $(I,f)$ is minimal. (a) If $h$ is a linear embedding then $|v_j|$ is independent of $j$. (b) If $h$ is an arc embedding then $r_j$ and $a_j$ are independent of $j$.
\end{lemma}

\begin{proof}
Assume without loss of generality that $I\subset [0,\pi)$. Note that there exists a linear embedding with rectangular atoms such that $T(x+iy)=f(x)+iy$, and there exists an arc embedding such that $T(re^{i\theta})=re^{i f(\theta)}$ .

We now prove (a) and (b). Fix $x \in I_p$ for some $p \in \{1,...,d\}$. Since $(I,f)$ is minimal, for all $q \in \{1,...,\hat{p},...,d \} $ there is a $N_q>0$ such that $f^{N_q}(x)=x+\tau \in I_q$, with $\tau=\sum_{k=0}^{N_q-1}\tau_{i_k(x)}$.

We begin by proving (a). Assume that $h$ is a linear embedding of $(I,f)$ into $(X,T)$  as in \eqref{eq1s2}. We show that $|v_p|=|v_q|$. By \eqref{eq0s2} and \eqref{eq1s2} we have
\begin{equation}\label{eq3s2}
e^{i \theta_p}(z_p+v_p x)+\lambda_p= z_q + v_q(x+\tau).
\end{equation}
Differentiating \eqref{eq3s2} with respect to $x$ gives $e^{i\theta_q} v_p = v_q$, thus $|v_p|=|v_q|$.

We now prove (b). Assume that $h$ is an arc embedding of $(I,f)$ into $(X,T)$ as in \eqref{eq2s2}. We show that $a_p=a_q$ and $r_p=r_q$. Combining \eqref{eq0s2}, \eqref{eq2s2} and differentiating with respect to $x$ we get
\begin{equation*}\label{eq4s2}
i r_p a_p \exp[i(\theta_p+a_p x+b_p)]=i r_q a_q \exp[i(a_q x+a_q \tau+b_q)],
\end{equation*}
and taking modulus gives
\begin{equation}\label{eq5s2}
r_p |a_p|= r_q |a_q|,
\end{equation}
while the argument gives
\begin{equation}\label{eq6s2}
\theta_p+a_p x+b_p=a_q x+a_q \tau+b_q \mod 2\pi.
\end{equation}
Note that \eqref{eq6s2} holds for any $x \in f^{-N_q}(I_q)\cap I_p$. Since this set contains an interval, \eqref{eq6s2} must hold for infinitely many values of $x$, hence we get $a_p=a_q$. Together with \eqref{eq5s2} this shows that $r_p=r_q$, completing the proof.
\end{proof}

The next theorem allow us to characterize the existence of continuous or discontinuous embeddings just in terms of the preimages of interior discontinuities of $f$.

\begin{theorem}
	Assume that $(I,f_{\mu,\pi})$ is a minimal $d$-IET with intervals $I_j=[x_{j-1},x_j)$ for $j=1,\ldots d$.  There exists a $d$-PWI $(X,T)$, such that $(I,f)$ has a discontinuous embedding into $(X,T)$ if and only if
	\begin{equation*}
		f_{\mu,\pi}^{-1}(\{x_1,...,x_{d-1}\})\cap \{x_1,...,x_{d-1}\} \neq \emptyset.
	\end{equation*}
\end{theorem}

\begin{proof}
	Let $I=I_1\cup...\cup I_d$, with
	$I_j=[x_{j-1},x_{j}),  j \in \{1,...,d\}.$
	
	We begin by proving that if there is  $j' \in \{1,...,d-1\}$ such that $f_{\mu,\pi}^{-1}(x_{j'})\in \{x_1,...,x_{d-1}\}$, then there exists a $d$-PWI $(X,T)$, such that $(I,f)$ has a discontinuous embedding into $(X,T)$.	
	
	By Lemma~\ref{l1s2} there is a continuous embedding of $(I,f)$ by $h'$ into a $d$-PWI $(X',T')$ with $Y'=h(I) \subset X'$ invariant set for $(X',T')$. Note that since this embedding is trivial we can take $X'$ to be a compact set. 
	
	Let $|I|$, $|X'|$ denote, respectively, the diameter of $I$ and $X'$. Set $Y'_j=Y'\cap X'_j$ for $j=1,...,d$ and let
	\begin{equation*}
	X_j=\left\{\begin{array}{ll}
	\vspace{0.2cm}
	X'_j, & \textrm{if} \  j\leq j',\\
	X'_j+2|X'|, & \textrm{if} \  j > j',\\
	\end{array}
	\right.
	\end{equation*}
	with
	$X=X_1\cup...\cup X_d.$
	Define the maps
	\begin{equation*}
	T_j(z)=\left\{\begin{array}{ll}
	\vspace{0.2cm}
	T'_j(z), &\textrm{if} \  j\leq j',\\
	T'_j(z-2|I|)+2|I|,   & \textrm{if} \  j > j'.\\
	\end{array}
	\right.
	\end{equation*}
	If $T(z)=T_j(z) ,$  for $z \in X_j, $ with  $j =1,...,d,$ then 
	 $(X,T)$ defines a $d$-PWI. 
	
	Define the function $h:I \rightarrow X$ as
	\begin{equation*}
	h(x)=\left\{\begin{array}{ll}
	\vspace{0.2cm}
	h'(x), & x <x_{j'},\\
	h'(x)+2|X'|, & x \geq x_{j'}.\\
	\end{array}
	\right.
	\end{equation*}
	Set $Y=h(I)$. The map $h:I \rightarrow Y$ is bijective and it is simple to check that $(I,f_{\mu,\pi})$  has a piecewise continuous embedding by $h$ into $(X,T)$. Moreover, note that the restriction of $h$ to $I_j$ is continuous for $j=1,...,d$, but $h$ has a discontinuity at $x=x_{j'}$. Thus, the embedding is discontinuous.

	Now assume there is no $x_j \in \{x_1,...,x_{d-1}\}$ such that $f_{\mu,\pi}^{-1}(x_j)\in \{x_1,...,x_{d-1}\}$ and there exists a $d$-PWI $(X,T)$, such that $(I,f_{\mu,\pi})$ has a discontinuous embedding by $h$ into $(X,T)$. 
	
	Since the restriction of $h$ to $I_j$ is continuous for all $j=1,...,d$,  the set of discontinuities of $h$ must be contained in $\{x_1,...,x_{d-1}\}$. Assume $j'\in \{1,...,d-1\}$ is such that $h$ is discontinuous at $x_{j'}$. Let
	$$
	\underline{z_{j'}} = \lim_{x \rightarrow x_{j'}^-}h(x), \quad \overline{z_{j'}} = \lim_{x \rightarrow x_{j'}^+}h(x)
	$$
and  $l\in \{1,...,d\}$ be such that $x_{j'} \in f_{\mu,\pi}(I_l)$. Set $Y=h(I)$ and $Y_j=X_j\cap Y$ for $j=1,...d$.  Then $\{\underline{z_{j'}},\overline{z_{j'}}\} \subset T(Y_l)$. Since $f_{\mu,\pi}^{-1}(x_{j'}) \notin \{x_1,...,x_{d-1}\}$, we have $$T^{-1}(\{\underline{z_{j'}},\overline{z_{j'}}\}) \cap \{h(x_1),...,h(x_{d-1})\}=\emptyset.$$ Thus there must be an $l'\in \{1,...,d\}$ such that $\{\underline{z_{j'}},\overline{z_{j'}}\} \subset Y_{l'}$. Therefore the restriction of $h'$ to $I_{l'}$ must be discontinuous, contradicting $h$ being a piecewise continuous embedding of $(I,f_{\mu,\pi})$ into $(X,T)$. This completes the proof.
\end{proof}

\vspace{0.2in}

We now extend Rauzy-Veech induction (see for instance \cite{AF}) for IETs to PWIs that admit continuous embeddings of IETs. Given a $d$-IET $f_{\mu,\pi}:I\rightarrow I$ such that $I_{d} \neq f(I_{\pi^{-1}(d)})$, we say that $f$ has \emph{type} $0$ if $f(I_{\pi^{-1}(d)}) \subset I_{d}$ and \emph{type} 1 if $I_{d} \subset f(I_{\pi^{-1}(d)})$.
	In both cases the largest interval is called \emph{winner} and the smallest \emph{loser}.
	
	Let $I'$ be interval obtained by removing the loser from $I$, that is
	\begin{equation}
	I'=
	\left\{
	\begin{array}{ll}
	\vspace{0.2cm}
	I \backslash f(I_{\pi^{-1}(d)}), & \textrm{if} \ f \ \textrm{has \ type} \ 0,\\
	I \backslash I_{d}, & \textrm{if} \ f \ \textrm{has \ type} \ 1.
	\end{array}
	\right.
	\end{equation}
	The Rauzy-Veech induction of $f$ is its first return map $\mathcal{R}(f)$ to $I'$, which again is a $d$-IET.
	
	We can extend this induction procedure to PWIs which admit continuous embeddings of IETs as follows. Assume $(I,f)$ has a continuous embedding by $h$ into $(X,T)$. Define the map $\mathcal{S}(T)$ as the first return map under $T$ to $X'$, where
	\begin{equation*}
	X'=\left\{\begin{array}{ll}
	\vspace{0.2cm}
	\bigcup_{j=1}^{d-1} X_j \cup (X_d \cap T(X_{\pi^{-1}(d)})) , & \textrm{if} \ f \ \textrm{has \ type} \ 0,\\
	\bigcup_{j=1}^{d-1} X_j , & \textrm{if} \ f \ \textrm{has \ type} \ 1.
	\end{array}\right.
	\end{equation*}
	$(X',\mathcal{S}(T))$ is again a $d'$-PWI since it is a first return map or a PWI to a convex subset of $X$. However it is now possible that $d'\neq d$. 
	
	We now show that a continuous embedding of $(I,f)$ into $(X,T)$ also embeds $(I',\mathcal{R}(f))$ into $(X',\mathcal{S}(T))$.
	
\begin{theorem}\label{RS}
	Assume that a  $d$-IET $(I,f)$, such that $I_{d} \neq f(I_{\pi^{-1}(d)})$, has a continuous embedding by $h$ into a $d$-PWI $(X,T)$. Then $(I',\mathcal{R}(f))$ has a continuous embedding by $h$ into $(X',\mathcal{S}(T))$.
\end{theorem}
\begin{proof}
	We prove that for all $x \in I'$ we have
	\begin{equation}\label{RS1}
	h \circ \mathcal{R}(f)  (x) = \mathcal{S}(T) \circ h (x).
	\end{equation}
	
	Assume first that $f$ has type $0$. Let $I_j'=I_j$ for $j \neq d$ and $I_d'=I_d \backslash f(I_{\pi^{-1}(d)})$. It is well known (see \cite{Vi06}) that 
	\begin{equation}\label{RS2}
	\mathcal{R}(f)(x)\left\{\begin{array}{ll}
	\vspace{0.2cm}
	f^2(x) , & x \in I_{\pi^{-1}(d)}',\\
	f(x) , & x \in I_{j}', \ j \neq \pi^{-1}(d).
	\end{array}\right.
	\end{equation} 
	We now show that we have
	\begin{equation}\label{RS3}
	\mathcal{S}(T)(z)\left\{\begin{array}{ll}
	\vspace{0.2cm}
	T^2(z) , & z \in h(I_{\pi^{-1}(d)}'),\\
	T(z) , & z \in h(I_{j}'), \ j \neq \pi^{-1}(d).
	\end{array}\right.
	\end{equation}

	Note that $f(I_j')\subset I'$, for $j \neq  \pi^{-1}(d)$. Thus,
	 by \eqref{eq0s2} we have $T(h(I_j'))\subset h(I')$, 
	and we get \eqref{RS3} for $z \in h(I_{j}')$ and $j \neq \pi^{-1}(d)$.

	Since $f(I_{\pi^{-1}(d)}')=f(I_{\pi^{-1}(d)}) \not \subset I'$ and
	$f^2(I_{\pi^{-1}(d)'}) \subset f(I_d) \subset I',$ 
	 by \eqref{eq0s2} we have $T(h(I_{\pi^{-1}(d)}'))=T(h(I_{\pi^{-1}(d)})) \not \subset h(I')$ and
	$T^2(h(I_{\pi^{-1}(d)}')) \subset T(h(I_d)) \subset h(I'),$
	and thus we have \eqref{RS3}.
	
	Noting that $x \in I_j$ if and only if $h(x) \in h(I_j')$, for $j =1,...,d$, and combining \eqref{eq0s2}, \eqref{RS2} and \eqref{RS3} we get \eqref{RS1}.
	
	Assume now that $f$ has type $1$. Let $I_j'=I_j$ for $1 \leq j < \pi^{-1}(d)$, $I_{\pi^{-1}(d)}'=I_{\pi^{-1}(d)}\backslash f^{-1}(I_d)$, $I_{\pi^{-1}(d)+1}'= f^{-1}(I_d)$ and $I_j'=I_{j-1}$ for $\pi^{-1}(d)+1 < j \leq d$. It is clear that
	\begin{equation}\label{RS6}
	\mathcal{R}(f)(x)\left\{\begin{array}{ll}
	\vspace{0.2cm}
	f^2(x) , & x \in I_{\pi^{-1}(d)+1}',\\
	f(x) , & x \in I_{j}', \ j \neq \pi^{-1}(d)+1.
	\end{array}\right.
	\end{equation} 
	By a similar argument it can be proved that
	\begin{equation}\label{RS7}
	\mathcal{S}(T)(z)\left\{\begin{array}{ll}
	\vspace{0.2cm}
	T^2(z) , & z \in h(I_{\pi^{-1}(d)+1}'),\\
	T(z) , & z \in h(I_{j}'), \ j \neq \pi^{-1}(d)+1.
	\end{array}\right.
	\end{equation}
	Since $x \in I_j$ if and only if $h(x) \in h(I_j')$, for $j =1,...,d$,  combining \eqref{eq0s2}, \eqref{RS6} and \eqref{RS7} we get \eqref{RS1}.
\end{proof}


\section{Connecting equations and continuous embeddings of $2,3$-IETs}
\label{sec:connecting}

In this section we  introduce a graph for a given permutation. We  use its combinatorial and topological properties to obtain a necessary condition for the parameters of a PWI to be a continuous embedding of an IET into a PWI described by the same permutation. 

We  then prove  that only trivial embeddings of $2$-IETs are possible and that a $3$-PWI has at most one non-trivially continuous embedded $3$-IET with the same underlying permutation. 

Given an irreducible permutation $\pi$ of $\{1,...,d\}$ and $\mu \in \mathbb{R}_+^d$, let $f_{\mu,\pi}(x):I \rightarrow I$ denote a minimal IET with $I=I_1 \cup ...\cup I_d$. As before we write $f=f_{\mu,\pi}$. Recall \eqref{eq3}.  It is clear that if $f_j(x)=x+\tau_j,$ for $ x \in I,$ and  $j=1,...,d$, then $f(x)=f_j(x),$  for $x \in I_j$.

For  $x_j$ with $0 \leq j \leq d$ as in \eqref{eq0} we have the following
\begin{equation}\label{ce5}
\begin{array}{l}
\vspace{0.3cm}
f_{\pi^{-1}(1)}(x_{\pi^{-1}(1)-1})=x_0, \\
\vspace{0.3cm}
f_{\pi^{-1}(j)}(x_{\pi^{-1}(j)-1})=f_{\pi^{-1}(j-1)}(x_{\pi^{-1}(j-1)}), \quad j=2,...,d,\\
 f_{\pi^{-1}(d)}(x_{\pi^{-1}(d)})=x_d.
\end{array}
\end{equation}

We extend $\pi$ to $\pi(0)=0$ and let $f_0$ be the identity map in $I$. For $j \in \Z$ we write $[j]= j \mod d+1$. We can write \eqref{ce5}  as 
\begin{equation}\label{ce7}
f_{\pi^{-1}([j])}(x_{[\pi^{-1}(j)-1]})
	=f_{\pi^{-1}([j-1])}(x_{\pi^{-1}([j-1])}), 
\end{equation}
where $j=0,...,d.$  

We now define a directed graph $\mathcal{G}_{\pi}$ in $d+1$ vertices $v_0,...,v_d$ such that there is an edge $$v_{\pi^{-1}([i-1])} \rightarrow v_{\pi^{-1}([j-1])}$$ if 
$\pi^{-1}([j-1]) = [\pi^{-1}(i)-1] ,$ with  $i,j \in\{0,...,d\}.$

The next proposition characterizes the topology of $\mathcal{G}_{\pi}$.
\begin{proposition}\label{tce1}
	Given an irreducible permutation $\pi$, the directed graph $\mathcal{G}_{\pi}$ is a disjoint union of directed cyclic subgraphs.
\end{proposition}
\begin{proof}
	Since $\mathcal{G}_{\pi}$ is a finite graph, it has a finite number of connected components, hence it suffices to prove that every connected component of $\mathcal{G}_{\pi}$ is a cyclic graph. 
	
	Consider a vertex $v_{q}$, with $q \in \{0,...,d\}$. There is a unique $i_0=[\pi(q)+1] \in \{0,...,d\}$, such that $\pi^{-1}([i_0-1])=q$.
	Define the map $\eta:\{0,...,d\}\rightarrow \{0,...,d\}$ as $\eta(n)=\pi ( [\pi^{-1}([n-1])+1] ).$ Note that $\eta$ is a bijection, hence $i_1=\eta(i_0)$  is the unique $i_1 \in \{0,...,d\}$ satisfying
	$$\pi^{-1}([i_0-1]) = [\pi^{-1}(i_1)-1].$$
Thus,  there is an edge $v_{q}\rightarrow v_{\pi^{-1}([i_1-1])}$.
	
	We now form a sequence  $\{i_k\}_{k\in \N}$ where $i_0=[\pi(q_1)+1]$ and $i_k=\eta(i_{k-1}),$ for $k \geq 1.$ 
	Since $\eta$ is a bijection between finite sets $\{i_k\}_{k\in \N}$, it must be a periodic sequence. If $\eta$ has period $d+1$, then $\mathcal{G}_{\pi}$ is a cyclic graph.
		Otherwise, $\eta$ has period $p \leq d$. This implies that the vertices $v_n$, for $n \in \{i_k\}_{0,...,p-1}$ and the edges connecting them form a connected and directed cyclic subgraph. Since the point  $q \in \{0,...,d\}$ was chosen without loss of generality, this shows that connected subgraphs of $\mathcal{G}_{\pi}$ are cycles. This completes the proof.
\end{proof}

\begin{proposition}\label{tce2}
	Let $\pi$ be an irreducible permutation of $\{1,...,d\}$, $\mu \in \mathbb{R}_+^d$ and $f_{\mu,\pi}(x):I \rightarrow I$ be a minimal IET. The directed graph $\mathcal{G}_{\pi}$ has an edge  $v_p\rightarrow v_q$ if and only if
	\begin{equation}\label{ce12}
	x_p=f_p^{-1}\circ f_{\pi^{-1}([\pi(p)+1])}(x_q).
	\end{equation}
\end{proposition}

\begin{proof}
	Let $p=\pi^{-1}([i-1])$ and $q=\pi^{-1}([j-1])$, for some $i, j \in \{0,...,d\}$. From \eqref{ce7} we have
	$f_{\pi^{-1}([i])}(x_{[\pi^{-1}(i)-1]})=f_{\pi^{-1}([i-1])}(x_{\pi^{-1}([i-1])}),$
which is equivalent to 
 	\begin{equation}\label{ce14}
 f_{\pi^{-1}([i])}(x_{\pi^{-1}([j-1])})=f_{\pi^{-1}([i-1])}(x_{\pi^{-1}([i-1])}),
 \end{equation}
 if and only if
 $\pi^{-1}([j-1])=[\pi^{-1}(i)-1],$
 that is, if  $v_p \rightarrow v_q$.
 From \eqref{ce14} we get \eqref{ce12}, which completes the proof.
\end{proof}

Now assume $(I,f)$ has a continuous embedding by $h$ into a $d$-PWI $(X,T)$ with $Y=h(I)$ and $Y_j=X_j\cap X$, such that
$T(z)= T_j(z),$ for $z \in Y_j,  j=1,..,d.$ with
\begin{equation}\label{ce16}
T_j(z)= e^{i\theta_j}z+\lambda_j, \quad z \in \C, \quad j=1,..,d.
\end{equation}
We call \eqref{ce16} the \emph{connecting equations}. Define $T_0$ as the identity map in $\C$. Let $z_j=h(x_j),$ for $j = 0,...,d$.
Equations \eqref{ce7} are preserved under topological conjugacy and can be written for $T$ as
\begin{equation}\label{ce17}
T_{\pi^{-1}([j])}(z_{[\pi^{-1}(j)-1]})=T_{\pi^{-1}([j-1])}(z_{\pi^{-1}([j-1])}), \quad j=0,...,d.
\end{equation}
The next corollary follows from Proposition \ref{tce2} and from the topological conjugacy of $(Y,T)$ and $(I,f)$.
\begin{corollary}\label{tce3}
	Assume a minimal $(I,f)$  has a continuous embedding by $h$ into a $d$-PWI $(X,T)$. The directed graph $\mathcal{G}_{\pi}$ has an edge  $v_p\rightarrow v_q$ if and only if
	\begin{equation*}\label{ce18}
	z_p=T_p^{-1}\circ T_{\pi^{-1}([\pi(p)+1])}(z_q).
	\end{equation*}
\end{corollary}

Let $p_0 \in \{0,...,d\}$. We define a \emph{connecting sequence} $\{p_k\}_{k \in \N}$ for $p_0$, with $p_k=q_{k-1}$, where $q_{k-1}$ is such that $v_{p_{k-1}}\rightarrow v_{q_{k-1}}$.
By Proposition \ref{tce1}, the connected component of $\mathcal{G}_{\pi}$ containing $v_{p_0}$ must be a directed cyclic graph. Thus,  $\{p_k\}_{k \in \N}$ is a well defined  periodic sequence with period $s(p_0) \leq d+1$. 

Furthermore define the \emph{connecting map} for $p_0$ as
	\begin{equation*}\label{ce19}
F_{p_0}(z)=T_{p_0}^{-1}\circ T_{\pi^{-1}([\pi(p_0)+1])}\circ...\circ T_{p_{s(p_0)-1}}^{-1}\circ T_{\pi^{-1}([\pi(p_{s(p_0)-1})+1])}(z), \quad z \in \C.
\end{equation*}

It follows from Corollary \ref{tce3}  that $z_{p_0}$ is a fixed point of $F_{p_0}$, thus, $F_{p_0}(z_{p_0})=z_{p_0}.$
We have
	\begin{equation}\label{ce20}
\left(e^{i\Theta_{\pi}(p_0)}-1\right)z_{p_0}+F_{p_0}(0)=0, 
\end{equation}
and
$$\Theta_{\pi}(p_0)=\sum_{k=0}^{s(p_0)-1}\theta_{\pi^{-1}([\pi(p_k)+1])}-\theta_{p_k}.$$

Now \eqref{ce20} either imposes a restriction on $h$, if $\Theta_{\pi}(p_0) \neq 0$, by forcing
\begin{equation}\label{ce21}
h(x_{p_0})=\left(1-e^{i\Theta_{\pi}(p_0) }\right)^{-1}F_{p_0}(0),
\end{equation}
or if $\Theta_{\pi}(p_0) = 0$ it imposes a restriction on the parameters $\lambda_j$, $\theta_j$, $j=1,...,d$, by
\begin{equation}\label{ce21a}
F_{p_0}(0)=0.
\end{equation}
Note that $F_{p_0}(0)$ can be seen as a sum where each term is $\lambda_j$ times a coefficient depending only on $\theta_1,...,\theta_d$. 

Denote the coefficient of $\lambda_j$ in $F_{p_0}(0)$  by $r_j(\theta_1,...,\theta_d)$ for $j=1,...,d$. Note that by linearity in $\lambda_j$, \eqref{ce21a} can be written as
\begin{equation}\label{ce22}
\sum_{j=1}^{d}\lambda_j r_j(\theta_1,...,\theta_d)=0.
\end{equation}
 We call  \eqref{ce22} the \emph{parametric connecting equation} for $p_0$.

\begin{theorem}\label{thm:necessary}
Assume a minimal $(I,f)$  has a continuous embedding by $h$ into a $d$-PWI $(X,T)$. If $\mathcal{G}_{\pi}$ is a connected graph, then the parameters $\lambda_j$, $\theta_j$, $j=1,...,d$ satisfy the parametric connecting equation \eqref{ce22}.
\end{theorem}
\begin{proof}
	Since $\mathcal{G}_{\pi}$ is a connected graph,  by Proposition \ref{tce1} it must be a directed cyclic graph. The connecting sequence for $p_0=0$ is well defined and has period $d+1$. Since the map $n \rightarrow \pi^{-1}([\pi(n)+1])$ is a bijection between finite sets we must have
	\begin{equation*}\label{ce23}
	\Theta_{\pi}(0)=\sum_{k=0}^{d}\theta_{\pi^{-1}([\pi(p_k)+1])}-\sum_{k=0}^{d}\theta_{p_k}=0.
	\end{equation*}
	Thus, there are functions $r_j(\theta_1,...,\theta_d)$ for $j=1,...,d$, not identically $0$, satisfying \eqref{ce22}.
\end{proof}

In the next theorem we prove that there are no non-trivial continuous embeddings of minimal 2-interval exchange transformations into orientation preserving planar PWIs.
\begin{theorem}\label{thm:trivial2IET}
	A minimal 2-IET  has no non-trivial continuous embedding into a $2$-PWI.
\end{theorem}

\begin{proof}
	Let $(I,f)$ be a minimal $2$-IET different from the identity with $\mu=\{\mu_1,\mu_2\} \in \R_+^2$. Assume there is a continuous embedding of $(I,f)$ by $h$ into a 2-PWI $(X,T)$ with partition $\{X_1,X_2\}$. 
	
	Set $Y=h(I)$ and $Y_j=Y\cap X_j$ for $j=1,2$.
	There are $\theta_j \in [0,2\pi)$ and $\lambda_j \in \C$, such that
	\begin{equation*}
	T_j(z)=e^{i \theta_j}z+ \lambda_j, \ z \in \C, \quad j=1,2,
	\end{equation*}
	and the restriction of $T$ to $Y$ is given by
	$$T(z)=T_j(z), \ z \in Y_j, \quad j=1,2.$$
	
	Since $f$ is not the identity, $\pi=(12)$ and  $\mathcal{G}_{\pi}$ is a connected graph, the connecting sequence for $p_0=0$ is $p=(0,1,2,...)$. This gives the connecting map 
	\begin{equation*}\label{ce23a}
	F_0(z)=T_0^{-1}\circ T_2 \circ T_1^{-1}\circ T_0 \circ T_2^{-1}\circ T_1 (z).
	\end{equation*}
	By Theorem \ref{thm:necessary}, the parameters $\lambda_1$, $\lambda_2$, $\theta_1$, and $\theta_2$ must satisfy the parametric connecting equation, which can be written as
	\begin{equation}\label{ce23b}
	\lambda_1(e^{-\theta_1}-e^{\theta_2-\theta_1})+\lambda_2(1-e^{-\theta_1})=0.
	\end{equation}
	Multiplying by $e^{i\theta_1}$, \eqref{ce23b} becomes
	\begin{equation}\label{nt7a}
	\lambda_2(1-e^{i \theta_1})=\lambda_1(1-e^{i \theta_2}).
	\end{equation}
	
	Since $T_j$ is not the identity map  \eqref{nt7a} is true if either both sides equal $0$ or not. 
	
	In the case that both sides are equal to zero, we have the following cases:
	
	i) If $\theta_1=\theta_2=0 \mod 2\pi$, then $T_j(z)=z+\lambda_j, \ z \in Y_j.$ Since we are assuming that $f$ is minimal and $Y$ is compact it follows that $T$ has dense orbits. This implies that there is  $s \in \R$ such that $\lambda_1=s\lambda_2$.
	For such a transformation, invariant sets must be unions of lines. This implies that $h$ is a trivial linear embedding.
	
	ii) If $\lambda_1=\lambda_2=0$, then $T_j(z)=e^{i\theta_j}z, \ z \in Y_j.$
	Since we are assuming that $f$ is minimal, the orbits of $T$ must be dense and in such a transformation, invariant sets must be unions of circle arcs. This implies that $h$ is a trivial circle arc embedding.
	
	iii) Finally, if $\lambda_j=0$ and $\theta_j=0 \mod 2\pi$, for $j=1$ or $2$ then $T_j$ is equal to the identity and hence $T$ can not be conjugated to a minimal IET.
	
	In the case that both sides of equation \eqref{nt7a} are different than $0$, there must exist $\lambda \in \C$ such that $\lambda_j = \lambda (1-e^{i\theta_j}), \ j=1,2.$
	This implies
	$$
	T_j(z)=(z-\lambda)e^{i\theta_j}+\lambda
	$$
	which is conjugate by  $L(z)=z+\lambda$, to the map
	$$
	\tilde{T}(z)=e^{i \theta_j}z, \ z \in Y_j-\lambda, \ j=1,2.
	$$
	and thus $h$ is a circle arc embedding. This completes the proof.
\end{proof}

In Section~\ref{sec:example} we present some numerical results which suggest that there exist non-trivial embeddings of $d$-IETs into $d$-PWIs, for $d=3$ and $d=4$.
The following example shows a permutation that yields a parametric connecting equation that can in principle allow the existence of non-trivial embeddings.

\begin{example}
	Consider the permutation $\pi=(2)(143)$. It is clear that  $\mathcal{G}_{\pi}$ is a connected graph. The connecting sequence for $0$ is $p=(0,2,3,1,4,...)$ and we have the connecting map 
\begin{equation*}\label{ce24}
F_0(z)=T_0^{-1}\circ T_{3} \circ T_{2}^{-1}\circ T_4 \circ T_3^{-1}\circ T_2 \circ T_1^{-1}\circ T_0 \circ T_4^{-1}\circ T_1 (z).
\end{equation*}
From this we get the following parametric connecting equation
\begin{equation}\label{ce25}
\lambda_1(e^{-i\theta_1}-e^{i(\theta_4-\theta_1)})+\lambda_2(e^{i(\theta_4-\theta_2)}-e^{i(\theta_3-\theta_2)})+\lambda_3(1-e^{i(\theta_4-\theta_2)})+\lambda_4(e^{i(\theta_3-\theta_2)}-e^{-i\theta_1})=0.
\end{equation}
\end{example}
In Section 5 we will discuss an example of a PWI satisfying \eqref{ce25}.	
Before proving the next theorem, recall that we are representing a permutation $\pi \in S(3)$ using cyclic notation.

\begin{theorem}\label{thm:3PWI}
	A $3$-PWI has at most one non-trivially continuously embedded minimal $3$-IET with the same underlying permutation.
\end{theorem}
\begin{proof}
	Given $\pi \in S(3)$ and $\mu \in \R_+^3$, assume there is a minimal $3$-IET $(I,f_{\mu,\pi})$ which is continuously embedded by $h$ into a $3$-PWI $(X,T)$, with  partition $\{X_1,X_2,X_3\}$ and
	$$T(z)= e^{i \theta_j} z + \lambda_j, \quad z \in X_j.$$
Let $Y=h(I)$. We show that $(I,f_{\mu,\pi})$ and $h$ are either unique or that the embedding is trivial.
	
	Assume first that $\pi=(123)$. Then $\mathcal{G}_{\pi}$ is not a connected graph. The connecting sequence for $1$ is constant and equal to $1$, thus, from \eqref{ce20} we get
	\begin{equation}\label{otiet1}
	(e^{i(\theta_2-\theta_1)}-1)h(x_1)+(\lambda_2-\lambda_1)e^{-i\theta_1}=0.
	\end{equation}
	We have $|\Theta_{\pi}(j)|=|\Theta_{\pi}|=|\theta_2-\theta_1|$ for $j =0,...,3$.
	
	If $\Theta_{\pi}=0$ then $\theta_1=\theta_2$, and by \eqref{otiet1} we get $\lambda_1=\lambda_2$.
	
	Consider the $2$-IET $(I,f_{\mu',\pi'})$, where  $\mu'=(\mu_1+\mu_2,\mu_3)$ and $\pi'$ is the permutation  $(12)$. Consider the $2$-PWI $(X,T')$, with base partition $\{X_1',X_2'\}$, where $X_1'=X_1 \cup X_2$ and $X_2'=X_3$ and
	$$T'(z)= e^{i \theta_j'} z + \lambda_j', \quad z \in X_j',$$
	with $\theta_1'=\theta_1$, $\theta_2'=\theta_3$, $\lambda_1'=\lambda_1$ and $\lambda_2'=\lambda_3$.
	It is simple to see now that $f_{\mu',\pi'}=f_{\mu,\pi}$ and $T'=T$, thus by Theorem \ref{thm:trivial2IET} the embedding of $(I,f_{\mu,\pi})$ must be trivial.
	
	If $\Theta_{\pi} \neq 0$, \eqref{ce21} gives
	\begin{equation}\label{otiet3}
	h(x_{j})=\left(1-e^{i\Theta_{\pi}(j) }\right)^{-1}F_{j}(0), \quad j=0,...,3.
	\end{equation}
	Since $F_{j}(0)$ does not depend of $\mu$, by \eqref{otiet3} we have that for any $\mu' \in \R_+^3$, such that $(I,f_{\mu',\pi})$ is minimal, any continuous embedding $h'$ into $(X,T)$ must satisfy $h'(x'_{j})=h(x_{j})$. Since the restriction of $T$ to $Y$ must be invertible and every $z \in Y$ must have a dense orbit in $Y$ this shows that $\mu'=\mu$ and $h'=h$.

	We omit the proof for $\pi=(321)$ as it can be done in a similar way to the previous case.
	
	Finally, assume that $\pi=(13)(2)$. Then $\mathcal{G}_{\pi}$ is not a connected graph. The connecting sequence for $1$ is equal to $(1,3,...)$, and from \eqref{ce20} we get
	\begin{equation}\label{otiet4}
	(\exp{\left[-i(\theta_3+\theta_1-\theta_2)\right]}-1)h(x_1)+e^{-i\theta_1}\left[e^{-i\theta_3}(\lambda_2-\lambda_3)-\lambda_1\right]=0.
	\end{equation}
	We have $|\Theta_{\pi}(j)|=|\Theta_{\pi}|=|\theta_3+\theta_1-\theta_2|$ for $j =0,...,3$.
	
	If $\Theta_{\pi}=0$ then by \eqref{otiet4} we get
	\begin{equation}\label{otiet5}
	\theta_2= \theta_1+ \theta_3, \quad \quad \lambda_2= \lambda_1 e^{i \theta_3}+\lambda_3.
	\end{equation}
	Note that $I_3=f_{\mu,\pi}(I_{\pi^{-1}(3)})$ if and only if $\mu_1=\mu_3$. In this case we have that the restriction of $f_{\mu,\pi}$ to $I_2$ is equal to the identity map. Since $f_{\mu,\pi}$ is minimal we must have $I_3\neq f_{\mu,\pi}(I_{\pi^{-1}(3)})$, thus by Theorem \ref{RS}  there is a continuous embedding of $(I', \mathcal{R}(f_{\mu,\pi}))$ by $h$ into $(X',\mathcal{S}(T))$.
	
	We now prove that this embedding is trivial. 
	
	Assume that $f_{\mu,\pi}$ has type 1. Let $I_j$ be as in the proof of Theorem \ref{RS}. By \eqref{RS7} we have
	\begin{equation}\label{otiet6}
	\mathcal{S}(T)(z)\left\{\begin{array}{ll}
	\vspace{0.2cm}
	e^{i\theta_1}z+\lambda_1 , & z \in h(I_1'),\\
	\vspace{0.2cm}
	e^{i(\theta_1+\theta_3)}z+(\lambda_1 e^{i \theta_3}+\lambda_3) , & z \in h(I_2'),\\
	e^{i\theta_2}z+\lambda_2 , & z \in h(I_3').	
	\end{array}\right.
	\end{equation}
	Consider the $2$-IET $(I,f_{\tilde{\mu},\tilde{\pi}})$, with  $\tilde{\mu}=(\mu_1-\mu_3,\mu_2+\mu_3), \tilde{\pi}=(12)$, and the map $\tilde{T}:h(I') \rightarrow h(I')$ such that 
	$$\tilde{T}(z)= e^{i \theta_j} z + \lambda_j, \quad z \in \tilde{Y}_j,$$
	where $\tilde{Y}_1=h(I_1')$ and $\tilde{Y}_2=h(I_2'\cup I_3')$.
	It is simple to see now that $f_{\tilde{\mu},\tilde{\pi}}=\mathcal{R}(f_{\mu,\pi})$ and by \eqref{otiet5} and \eqref{otiet6} we have $\tilde{T}(z)=\mathcal{S}(T(z))$, for all $z\in h(I')$. Therefore by Theorem \ref{thm:trivial2IET} the embedding of $(I', \mathcal{R}(f_{\mu,\pi}))$ by $h$ into $(X',\mathcal{S}(T))$ must be trivial. By \eqref{eq0s2} we have that for $x \in I_3$ we have $h(x)=e^{i\theta_1}h(x-\mu_2-\mu_3)+\lambda_1$ thus the embedding of $(I,f_{\mu,\pi})$ by $h$ into $(X,T)$ must be trivial as well.
	
	We omit the proof for the case when $f_{\mu,\pi}$ has type 0 as it can be done in a similar case to the previous case.
	
	Finally, if $\Theta_{\pi} \neq 0$, by \eqref{ce21}, $h(x_{j})$ is determined by \eqref{otiet3}. Since $F_{j}(0)$ does not depend of $\mu$, we have that for any $\mu' \in \R_+^3$, such that $(I,f_{\mu',\pi})$ is minimal, any continuous embedding $h'$ into $(X,T)$ must satisfy $h'(x'_{j})=h(x_{j})$. Since the restriction of $T$ to $Y$ must be invertible and every $z \in Y$ must have a dense orbit in $Y$ this shows that $\mu'=\mu$ and $h'=h$.
\end{proof}

\section{Ergodic condition for the existence of piecewise continuous embeddings}
\label{sec:ergodic}

In this section we give a necessary condition for the existence of piecewise continuous embeddings of uniquely ergodic IETs into planar PWIs. 


Given a $d$-IET $(I,f_{\mu,\pi})$, let $I_j$ and $\tau_j$ be as in Section~1. Suppose we have a piecewise continuous embedding $h$ of this map into a $d$-PWI $(X,T)$ and suppose that $T(z)=T_j(z)$, for $z \in X_j$ with  $T_j(z)=e^{i\theta_j}z+\lambda_j$. 

For $x \in I$ and $y \in S^1$ we define the  \emph{tangent exchange map} $\Psi: I \times S^1 \rightarrow I \times S^1$ as the skew-product given by
\begin{equation}\label{eq9}
\Psi(x,y)=(f(x),y+\theta_{j(x)}).
\end{equation}

The dynamics of this map contains information on the angle of tangents of an embedding when iterated by the underlying PWI. It will be the main technical tool to prove Theorem \ref{t:ergodiccondition}. 

For $n \in \N$ we have
\begin{equation*}\label{eq9a}
\Psi^n(x,y)=(f^n(x),y+C^{(n)}(x)),
\end{equation*}
where $C^{(\cdot)}:\Z \times I \rightarrow S^1$ is the \emph{rotational cocycle} for this embedding,  given by
\begin{equation*}\label{eq8}
C^{(0)}(x)=0, \quad C^{(n)}(x)=\theta_{j(x)}+...+\theta_{j(f^{n-1}(x))} \mod 2\pi, 
\end{equation*}
for $x \in I, n\geq 0,$
and $$C^{(n)}(x) = -C^{(-n)}(x) \mod 2\pi,$$ 
for $n<0$, where $S^1=\R/2\pi\Z$ and $j(x)$ is the piecewise constant map such that $j(x)=j$ when $x\in I_j$. 

For $x \in I_j$ we define the {\it first return time of $x$ by $f_{\mu,\pi}$ to $I_j$} as
\begin{equation*}\label{eq10}
n_j(x)= \inf\{k>1: f_{\mu,\pi}^k(x)\in I_j\}.
\end{equation*}
If $f_{\mu,\pi}$ is minimal, then $n_j(x)$ is finite. The {\it first return map of $x$ by $f$ to $I_j$}, $f'_{j}:I_j\rightarrow I_j$ is then a well defined $d$-IET and is given by
\begin{equation}\label{eq10a}
f'_{j}(x)= f^{n_j(x)}(x), \quad x \in I_j.
\end{equation}
For $j=1,...,d$, we define the cocycle $N_j^{(\cdot)}:\Z\times I_j \rightarrow \Z$ as 
\begin{equation*}\label{eq10b}
N_j^{(0)}(x)=0, \quad N_j^{(k)}(x)= n_j\left( x \right)+n_j\left( f_j^{'}(x)\right)+...+n_j\left( f_j^{' k-1}(x)\right),
\end{equation*}
for $x \in I_j$ and $k>0$. For $n<0$ we set $N_j^{(k)}(x) = -N_j^{(-k)}(x)$.


\begin{figure}[t]
	\begin{subfigure}{.49\textwidth}
		\centering
		\includegraphics[width=0.5\linewidth]{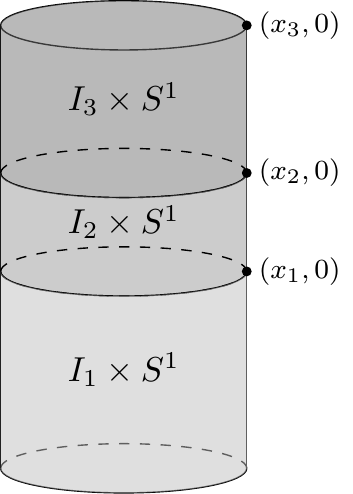}
		\caption{}
		\label{fig:bonecos2}
	\end{subfigure}
	\begin{subfigure}{.49\textwidth}
		\centering
		\includegraphics[width=0.55\linewidth]{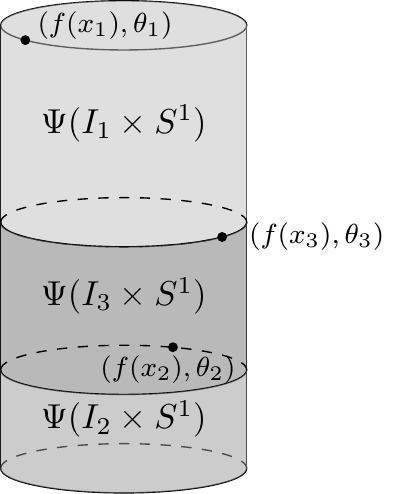}
		\caption{}
		\label{fig:bonecos2a}
	\end{subfigure}
	\centering
	\captionsetup{width=1\textwidth}
	\caption{A schematic representation of the action of a tangent exchange map $S$, as in \eqref{eq9}, on a cylinder with $\pi=(132)$. (A) The partitioned space $ I\times S^1$ in three subcylinders $I_j \times S^1$. The $x_j$ are equal to $\sum_{i\leq j}\mu_i$ respectively for $j=1,2,3$ and the points  $(x_j,0)$ are represented. (B) The action of the map $\Psi$ on $ I\times S^1$ and on the points $(x_j,0)$ which map to $(f_{\mu,\pi}(x_j),\theta_j)$ respectively for $j=1,2,3$.}
	\label{fig:cylinder}
\end{figure}

Define the sequence $\{p(n)\}_{n\geq1}$ by
\begin{equation*}\label{eq18}
p(1)=\min\{p\geq 1 : f_{\mu,\pi}^p(0) \in I_1\},
\end{equation*}
and
\begin{equation*}\label{eq19}
p(n)=\min\{p> p(n-1): f_{\mu,\pi}^p(0)<f_{\mu,\pi}^{p(n-1)}(0) \}, \quad n>1.
\end{equation*}
Note that if $f_{\mu,\pi}$ is minimal then $f_{\mu,\pi}^{p(n)}(0)\rightarrow 0$, as $n \rightarrow +\infty$.
Let
$$m_j(n)=\textrm{card}\{f_{\mu,\pi}^k(0)\in I_j:k\leq n\},$$ 
with $n \in \N, j=1,...,d,$ and
$$k_j=\min\left\{k\geq0: f_{\mu,\pi}^{k}(0)\in I_j  \right\}.$$

Denote  $x_j'=f^{k_j}(0)$, $y_j'=C^{(k_j)}(0)$. For $n \in \N$ and $j=1,...,d$, 
define the sequences
\begin{equation*}\label{eq12}
c_j(n)=y_j'+C^{\left(N_j^{(n)}(x)+1\right)}(x_j'),
\end{equation*}
and
\begin{equation}\label{eq22}
e_j(n)=\sum_{k=0}^{m_j(n)-1}\exp[-i c_j(k)].
\end{equation}
The sequence $e_j(p(n))$ is a cumulative sum of the displacement by rotation of the point $x_j'$, up to the $n$-th return to $X_j$. The limit of $e_j(p(n))$, when $n\rightarrow + \infty$, need not exist in general.
The following theorem shows that for a piecewise continuous embedding of a uniquely ergodic $(I,f_{\mu,\pi})$, as long as the average of the sequence $\{e_j(p(n))\}_n$ converges, the condition \eqref{eq25} tells us that the average of displacements by rotation and by translations, weighted by the lengths $\mu_j$'s, must cancel out so that orbits remain bounded.

\begin{theorem}\label{t:ergodiccondition}
	Assume that $(I,f_{\mu,\pi})$ is a uniquely ergodic  $d$-IET that has a piecewise continuous embedding by $h$ into a $d$-PWI $(X,T)$ with $X\subseteq \C$, where
	\begin{equation}\label{eq13}
	T(z)= e^{i\theta_j}z+\lambda_j, 
	\end{equation}
	for $z \in X_j$ and  $j=1,..,d$. If there are $\xi_j \in \C$ such that
	\begin{equation}\label{eq24}
	\lim_{n\rightarrow + \infty}\frac{1}{m_j(p(n))}e_j(p(n))=\xi_j,
	\end{equation}
	then 
	\begin{equation}\label{eq25}
	\sum_{j=1}^{d}\left(\lambda_j-h(0)(1-e^{i \theta_j})\right) \xi_j \mu_j=0.
	\end{equation}
\end{theorem}

\begin{proof}
	We begin by proving that there is an orientation preserving PWI $(\tilde{X},\tilde{T})$, conjugated by a translation to $(X,T)$, such that $(I,f_{\mu,\pi})$ has a piecewise continuous embedding by $\tilde{h}$ into $(\tilde{X},\tilde{T})$ with $\tilde{h}(0)=0$.
	
	Let $\tilde{X}=\{z \in \C: z+h(0) \in X\}$, and $q:X \rightarrow \tilde{X}$ be such that $q(z)=z-h(0)$. Let
	\begin{equation*}
	\tilde{T}(z)=q\circ T \circ q^{-1}(z),
	\end{equation*}
for $z \in \tilde{X}$. The homeomorphism $\tilde{h}= q \circ h$  conjugates $(I,f)$ to $(\tilde{h}(I),\tilde{T})$, with $\tilde{h}(I)\subseteq\tilde{X}$ invariant for $(\tilde{X},\tilde{T})$. Moreover, $\tilde{h}(0)=q(h(0))=0$. Note that we have
	\begin{equation*}
	\tilde{T}(z)=e^{i \tilde{\theta}_j}z+\tilde{\lambda}_j, 
	\end{equation*}
for $z \in \tilde{X}_j,$	 where $\tilde{X}_j=\{z \in \C: z+h(0) \in X_j\}$, $\tilde{\theta_j}=\theta_j$ and $\tilde{\lambda_j}=\lambda_j-h(0)(1-e^{i \theta_j})$.

	We now prove that 
	\begin{equation}\label{eq23}
	\lim_{p\rightarrow + \infty}\sum_{j=1}^{d}\tilde{\lambda_j} e_j(n(p))=0.
	\end{equation}

	Since $(I,f_{\mu,\pi})$ has a piecewise continuous embedding by $\tilde{h}$ into $(\tilde{X},\tilde{T})$, we have 
	\begin{equation}\label{eq14}
	\tilde{h}(x+\tau_j)= e^{i\theta_j}\tilde{h}(x)+\tilde{\lambda}_j, 
	\end{equation}
for $x \in I_j, j=1,..,d.$ Let $\tilde{Y}=\tilde{h}(I)$, $\tilde{Y}_j=\tilde{Y} \cap \tilde{X}_j$ and $\tilde{h}_j:I_j \rightarrow \tilde{Y}_j$ be the restriction of $\tilde{h}$ to $I_j$. From \eqref{eq14} we get 
	\begin{equation*}\label{eq15}
	\tilde{h}_j(x)= e^{-i\theta_j}(\tilde{h}_k(x+\tau_j)-\tilde{\lambda}_j), 
	\end{equation*}
where $x \in f_{\mu,\pi}^{-1}(I_k),$ and  $j=1,..,d.$	
	
	Recall the itinerary of $x$ as in \eqref{eq6}. It can be proved by induction that for $ x \in I,  n\in \N$ we have
	\begin{equation}\label{eq16}
	\tilde{h}_{i_0}(x)= \exp\left[-i\sum_{k=0}^{n-1}\theta_{i_k}\right]\tilde{h}_{i_n}(f_{\mu,\pi}^n(x))-\sum_{k=0}^{n-1}\tilde{\lambda}_{i_k}\exp\left[-i\sum_{l=0}^{k}\theta_{i_l}\right], 
	\end{equation}
	
	Since  $\tilde{h}(0)=0$, taking $x=0$ in \eqref{eq16} we get
	\begin{equation}\label{eq17}
	\exp\left[-i\sum_{k=0}^{n-1}\theta_{i_k}\right]\tilde{h}_{i_n}(f_{\mu,\pi}^n(0))-\sum_{k=0}^{n-1}\tilde{\lambda}_{i_k}\exp\left[-i\sum_{l=0}^{k}\theta_{i_l}\right]=0 , 
	\end{equation}
for $n\in \N$.	
	
	Note that $\tilde{h}_j: I_j \rightarrow \tilde{Y}_j$ is a homeomorphism for $j=1,...,d$. By  continuity of $\tilde{h}_1$ and \eqref{eq17} 
	\begin{equation}\label{eq20}
	\left| \tilde{h}_1(f_{\mu,\pi}^{p(n)}(0))-\tilde{h}_1(0)\right| = \left| \sum_{k=0}^{p(n)-1}\lambda_{i_k}\exp\left[-i\sum_{l=0}^{k}\theta_{i_l}\right] \right| \xrightarrow{n\rightarrow+\infty} 0.
	\end{equation}
	By \eqref{eq22}, \eqref{eq20} is equivalent to \eqref{eq23}.
	
	We now show that 
	\begin{equation}\label{eq25a}
	\sum_{j=1}^{d}\tilde{\lambda}_j\xi_j \mu_j=0.
	\end{equation}
	Since $(I,f_{\mu,\pi})$ is uniquely ergodic with respect to Lebesgue measure, 
	\begin{equation}\label{eq26}
	\lim_{n\rightarrow + \infty}\frac{m_j(p(n))}{p(n)}=\frac{\mu_j}{|I|}, 
	\end{equation}
for $j=1,...,d.$	
	
	Note that \eqref{eq24} is equivalent to
	\begin{equation}\label{eq27}
	e_j(p(n))=m_j(p(n))\xi_j+o(p(n)), \quad j=1,...,d.
	\end{equation}
	
	Combining \eqref{eq26} and \eqref{eq27} we have
	\begin{equation*}\label{eq28}
	e_j(p(n))= p(n) \frac{m_j(p(n))}{p(n)}\frac{1}{m_j(p(n))}e_j(p(n))=(p(n)+o(p(n)))\frac{\mu_j}{|I|} \xi_j, 
	\end{equation*}
for $j=1,...,d,$	
	and we get
	\begin{equation}\label{eq29}
	\sum_{j=1}^{d}\tilde{\lambda}_j e_j(p(n))=\sum_{j=1}^{d}(p(n)+o(p(n)))\tilde{\lambda}_j  \mu_j \xi_j.
	\end{equation}
	
	Since $(I,f_{\mu,\pi})$  has a piecewise continuous embedding into $(X,T)$, \eqref{eq23} holds. Thus \eqref{eq29} implies that
	\begin{equation*}\label{eq30}
	\lim_{n\rightarrow + \infty}\sum_{j=1}^{d}(p(n)+o(p(n)))\tilde{\lambda}_j  \mu_j \xi_j=0,
	\end{equation*}
	which can only hold if \eqref{eq25a} is true, as desired. Finally note that \eqref{eq25a} is equivalent to \eqref{eq25}, and the proof is complete.
\end{proof}

Condition \eqref{eq24} is not simple to verify in general since $c_j(n)$ is determined by two cocycles. However under some assumption on $\theta_j$ we can identify $c_j(n)$ with an orbit of a point by interval exchange map and  compute the $\xi_j$ as spatial averages using the ergodic theorem.

\begin{corollary}\label{t2}
	Assume that $(I,f_{\mu,\pi})$ is a uniquely ergodic $d$-IET with a piecewise continuous embedding by $h$ into a $d$-PWI $(X,T)$  as in \eqref{eq13}. Let $\chi_{I_j}$ denote the characteristic function of $I_j$. If 
	\begin{equation}\label{eq31}
	\theta_j=\frac{2\pi}{|I|}\tau_j,
	\end{equation}
	for $j=1,...d$, then 
	\begin{equation}\label{eq31a}
	\int_{I}\left(\sum_{j=1}^{d} \left(\lambda_j-h(0)(1-e^{i \theta_j})\right)\chi_{I_j}(f^{-1}(x))\right)e^{-2\pi i x}dx=0.
	\end{equation}
\end{corollary}

\begin{proof}
	Let $f_j':I_j\rightarrow I_j$ be as in \eqref{eq10a}. By \eqref{eq31} we have 
	\begin{equation*}\label{eq33}
	c_j(n)=\frac{2\pi}{|I|}f \circ f_j^{'n}(x_j'),
	\end{equation*}
	 Since $(I,f_{\mu,\pi})$ is uniquely ergodic, it follows that $(I_j,f_j')$ is also uniquely ergodic. Thus, the ergodic theorem implies that
	\begin{equation}\label{eq34}
	\lim_{n\rightarrow + \infty}\frac{1}{m_j(p(n))}\sum_{k=0}^{m_j(p(n))-1}\exp\left[-\frac{2\pi i}{|I|}f \circ f_j^{'k}(x_j')\right] = \frac{1}{\mu_j}\int_{f(I_j)}\exp\left[-2\pi i x\right]dx, 
	\end{equation}
for $j=1,...,d.$	
	
Let $\xi_j=\frac{1}{\mu_j}\int_{f(I_j)}\exp\left[-2\pi i x\right]dx$, for $j=1,...,d$. Combined with \eqref{eq22} and \eqref{eq34} we get
	\begin{equation*}\label{eq35}
	\lim_{n\rightarrow + \infty}\frac{1}{m_j(p(n))}e_j(p(n))=\xi_j, 
	\end{equation*}
for $j=1,...,d,$	and thus by Theorem \ref{t:ergodiccondition} we must satisfy \eqref{eq25} which is equivalent to \eqref{eq31a}.	This completes the proof.
\end{proof}

\section{Evidence of non-trivial embeddings of IETs into PWIs}
\label{sec:example}

In this section we present some numerical evidence of non-trivial continuous embeddings of IETs in PWIs. In order to do this we first define a PWI on 3 atoms that apparently exhibits a single invariant curve that is the image of a non-trivial embedding of a $3$-IET. We also introduce a new family of PWIs that include linear embeddings of $2$-IETs and apparently many non-trivial embeddings of $4$-IETs.

\subsection{A PWI with an embedded three interval exchange} We now present an example of a $3$-PWI for which numerical evidence suggests the existence of a non-trivial embedded $3$-IET.

Let $\alpha'=1.3$, $\beta'=0.75$, $z_0'=0$, $z_1'=0, 0.215998 + i 0.168125$, $z_2'=0.491520 + i 0.051612$, $z_3'=0.586452$ and the convex sets
\begin{equation*}
\begin{array}{l}
\vspace{0.2cm}
Q_1'=\{z\in \C: \Im(e^{i\alpha'}(z-z_1'))<0 \},\\
\vspace{0.2cm}
Q_2'=\{z\in \C: \Im(e^{-i\beta'}(z-z_2')) > 0 \ \textrm{and} \ \Im(e^{i\alpha'}(z-z_1')) \geq 0 \},\\
\vspace{0.2cm} 
Q_3'=\{z\in \C:  \Im(e^{-i\beta'}(z-z_2')) \leq 0 \ \textrm{and} \ \Im(e^{i\alpha'}(z-z_1')) \geq 0 \}.
\end{array}
\end{equation*}

Consider the PWI $T':\C \rightarrow \C$ such that
\begin{equation}\label{Tprime}
T'(z)=e^{i \theta_j'}z +\lambda_j' ,  \quad z \in Q_j',
\end{equation}
for $j=1,2,3$, where
\begin{equation}\label{tetlambda}
\begin{array}{ll}
\theta_j'=\left\{\begin{array}{ll}
\vspace{0.2cm}
4.460361, & j=1,\\
\vspace{0.2cm}
0.800153, & j=2,\\
0.995933, & j=3,
\end{array}\right.
\\
\\
\begin{array}{ll}
\lambda_j'=\left\{\begin{array}{ll}
\vspace{0.2cm}
z_3'-e^{i\theta'_1}z_1', & j=1,\\
\vspace{0.2cm}
e^{i\theta'_3}(z_3'-z_2')-e^{i\theta'_2}z_1', & j=2,\\
e^{i\theta'_3}z_2', & j=3,
\end{array}\right.
\end{array}
\end{array}
\end{equation}
and set $Y'=\overline{\{T'^n(z_0')\}_{n\in \N}}$. Figure \ref{fig:3PWI} shows the action of the map $T'$, in particular in Figure \ref{fig:3PWI} (A) we can see $Y'$ and in Figure \ref{fig:3PWI} (B) its image by $T'$.
\begin{figure}[t]
	\begin{subfigure}{0.49\textwidth}
		\centering
		\includegraphics[width=1\linewidth]{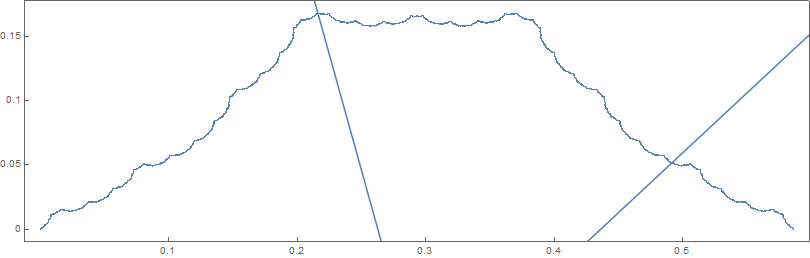}
		\caption{}
		\label{fig:3PWIa}
	\end{subfigure}
	\begin{subfigure}{0.49\textwidth}
		\centering
		\includegraphics[width=1\linewidth]{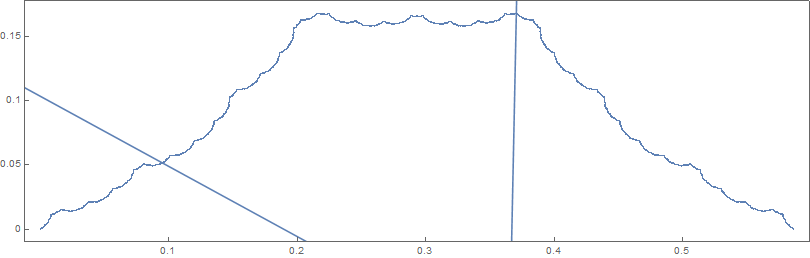}
		\caption{}
		\label{fig:3PWIb}
	\end{subfigure}\\
	\begin{subfigure}{0.49\textwidth}
	\centering
	\includegraphics[width=1\linewidth]{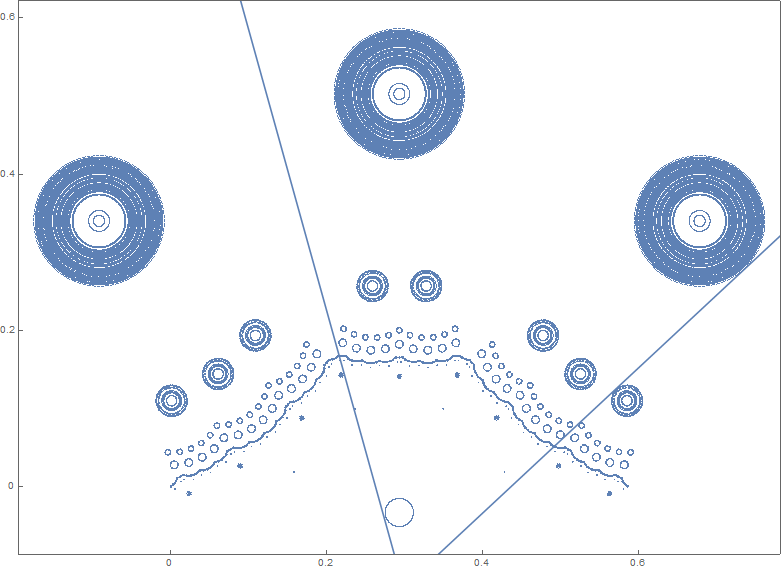}
	\caption{}
	\label{fig:3PWIc}
	\end{subfigure}
	\begin{subfigure}{0.49\textwidth}
	\centering
	\includegraphics[width=1\linewidth]{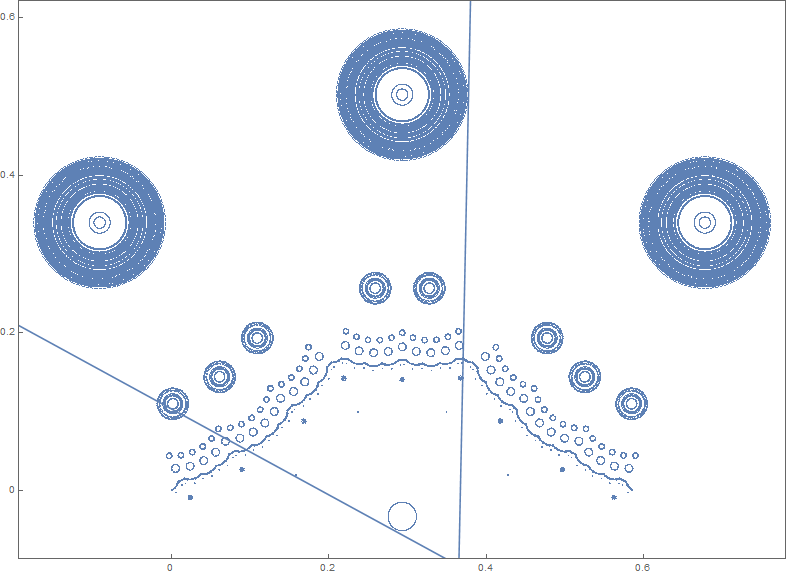}
	\caption{}
	\label{fig:3PWId}
	\end{subfigure}\\
	\centering
	\captionsetup{width=1\textwidth}
	\caption{An illustration of the action of the piecewise isometry $T'$.
		(A) An invariant set $Y'$  and the partition $\{Q_j'\}_{j=1,2,3}$. (B) Image of $Y'$ by $T'$.
	(C) Orbits of 40 points, including $z_0$, (ignoring a transient) under $T'$  and the partition $\{Q_j'\}_{j=1,2,3}$. (D) Image of the orbits and the partition in (C) by $T'$.}
	\label{fig:3PWI}
\end{figure}
Consider the family $\mathcal{F}_3$ of 3-IETs $f_{\mu,\pi'}:I\rightarrow I$ given by subdividing the interval into four intervals of lengths $\mu=(\mu_1,\mu_2,\mu_3) \in \R_+^4$ with base permutation $\pi'=(2)(13)$.

We can partition $Y'$ by setting $Y_j'=Y' \cap Q_j'$, for $j=1,2,3$. The length $l_j'=Leb(Y_j')$ of each $Y_j'$ can be numerically estimated to be $$l_1' = 0.3910666426, l_2' =0.4553369973, l_3' = 0.1535963601.$$ Fix $\mu = (l_1',l_2',l_3')$ and consider the IET $(I,f_{\mu,\pi'}) \in \mathcal{F}_3$.
Numerical evidence suggests that there is a continuous embedding of $(I,f_{\mu,\pi'})$ into $(\C,T')$, by a map $h':I\rightarrow Y'$ with $Y' \subseteq \C$, such that $h(0)=z_0'$. Note that $\mathcal{G}_{\pi'}$ is not a connected graph so we are not in the conditions of Theorem \ref{thm:necessary}. However by \eqref{tetlambda} it is simple to check that  \eqref{ce17} and \eqref{ce21} are satisfied.
 Indeed numerical verification shows that $i_k'{R(h(0))}=i_k(f_{\mu,\pi'}(0))$ for all $k \leq 6 \times 10^4$, supporting that $h'$ is a symbolic embedding.

We can also verify numerically that the condition in Theorem \ref{t:ergodiccondition} holds for this case. We estimate $\xi_j' \simeq \frac{e_j(p(8))}{m_j(p(8))}$ where $\xi_1' \simeq -0.453 + 0.651 i$, $\xi_2' \simeq 0.326 + 0.669 i$ and $\xi_3' \simeq 0.417 + 0.679 i$. For these estimates we get
$$
\left|\sum_{j=1}^{d}\lambda_j'\xi_j' \mu_j - h'(0)\sum_{j=1}^{d}(1-e^{i\theta_j'})\xi_j' \mu_j\right| \simeq 1.19 \times 10^{-5}.
$$

\subsection{A planar piecewise isometry with four cones} We now  introduce a new family of PWIs that include a linear embedding of a $2$-IET and, apparently an infinite number of non-trivial embeddings of $4$-IETs.

For any $\beta \in (0,\frac{\pi}{2})$ and $\alpha \in (0,\pi-2\beta)$ and $\lambda \in \mathbb{R}^+$ we consider a partition of $\C$ into four atoms 
$$\begin{array}{l}
P_0=\{ z \in \C: \arg(z) \in [-\beta,\beta) \}\cup \{0\},\\
P_1=\{ z \in \C: \arg(z) \in [\beta,\alpha+\beta) \},\\
P_2=\{ z \in \C: \arg(z) \in [\alpha+\beta,\pi-\beta) \},\\
P_3=\{ z \in \C: \arg(z) \in [\pi-\beta,2\pi-\beta) \},
\end{array}$$
and define a map $T:\mathbb{C} \rightarrow \mathbb{C}$ by $T(z)=T_j(z)$, for $z \in P_j$, where
\begin{equation}\label{e:pwi4}
T_j(z)=\begin{cases}
z-1,  & z \in j=0,\\
ze^{i\vartheta_1} -(1-\lambda), &z \in j=1,\\
z e^{i \vartheta_2} -(1-\lambda), &z \in j=2,\\
z+ \lambda, &z \in j=3,
\end{cases}
\end{equation}
and $\vartheta_1= \pi - 2 \beta - \alpha,$ $\vartheta_2=-\alpha$. An example is illustrated in Figure~\ref{fig_iet_cones}: note that this map is not invertible. We define the maximal invariant set for this map as $X\subset \C$.  Note that $T$ restricted to the real line reduces to a $2$-IET on 
$[-1,\lambda)$ that equivalent to interchange of intervals of length $1$ and $\lambda$. We refer to this as the {\em baseline} transformation.

\begin{figure}[t]
	\mbox{\includegraphics[width=12cm]{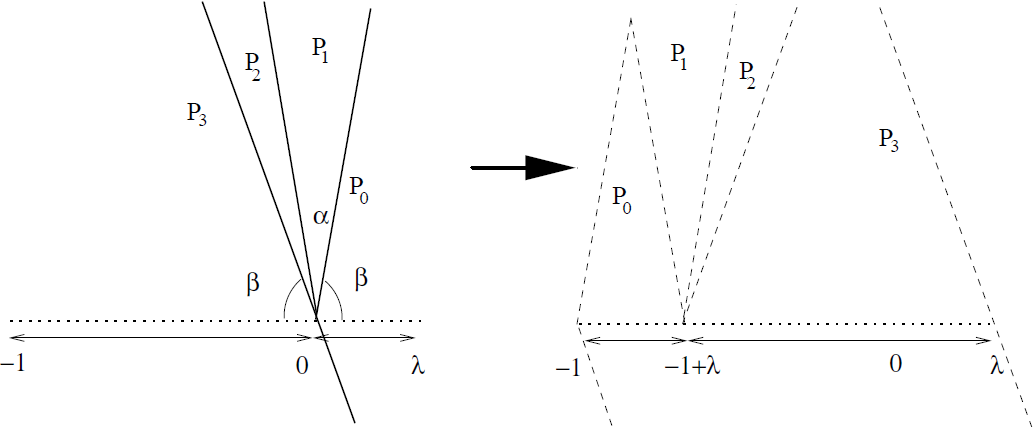}}
	\captionsetup{width=1\textwidth}
	\caption{
		\label{fig_iet_cones}
		Schematic representation of a family of $4$-PWIs $T:\C\rightarrow \C$ with atoms given by the four cones $P_i$, and three parameters: $\alpha,\beta$ and $\lambda$. The atoms $P_0$ and $P_3$ are translated by $T$ while $P_1$ and $P_2$ are rotated about their vertices then translated. The map on the baseline $[-1,\lambda)$ is a $2$-IET.
	}
\end{figure}

This map is such that all vertices of atoms that touch the baseline are mapped to the baseline. This means that although $T$ is not invertible, it is locally bijective near the base line. The middle cones $P_1$ and $P_2$ are swapped by two rotations and after this, $P_1$ and  $P_2$ are translated by $-(1-\lambda)$.

%
%
%
%


We define the first return map $R: P_1 \cup P_2 \rightarrow P_1 \cup P_2$, as
\begin{equation}\label{n1}
  R(z)=T^{k(z)}(z).
\end{equation}
where $k(z)=\inf\{k\geq 1: T^{k}(z) \in P_1 \cup P_2\}$. If $\lambda$ is irrational then every point enters $P_1\cup P_2$ after a finite number of iterates, and hence in this case $R$ can be used to characterise all orbits of the map.


\begin{figure}[t]
	\begin{subfigure}{.45\textwidth}
		\centering
		\includegraphics[width=1\linewidth]{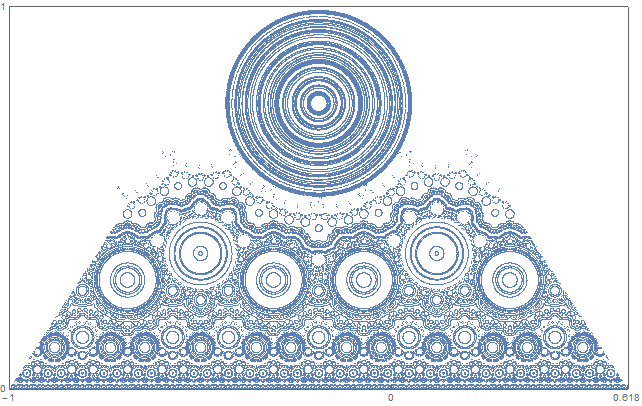}
		\caption{}
		\label{fig:alpha05beta1lambdagoldengamma0}
	\end{subfigure}
	\begin{subfigure}{.45\textwidth}
		\centering
		\includegraphics[width=1\linewidth]{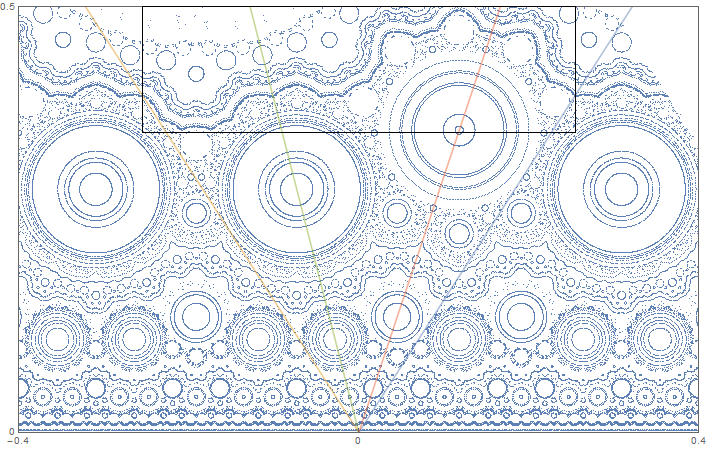}
		\caption{}
		\label{fig:renormgolden}
	\end{subfigure}
	\centering
	\captionsetup{width=1\textwidth}
	\caption{(A) Orbits of 200 points (ignoring a transient) by $T$, for $(\alpha, \beta, \lambda)= (0.5, 1,\frac{\sqrt{5}-1}{2})$. (B) Details of (A) in the area $[-0.4,0.4]\times[0,0.5]$. The cone indicates the location of $P_1\cup P_2$. In this and later figures, orbits of length $10^5$ are generated after removing a transient of 100 iterates. The maximal invariant set appears to have a highly complex boundary, but it does appear to include a polygon containing the baseline. The boxed region contains what seem to be many invariant non smooth curves.}
	\label{fig:Attractors1}
\end{figure}

\begin{figure}[t]
	\begin{subfigure}{.7\textwidth}
		\centering
		\includegraphics[width=1\linewidth]{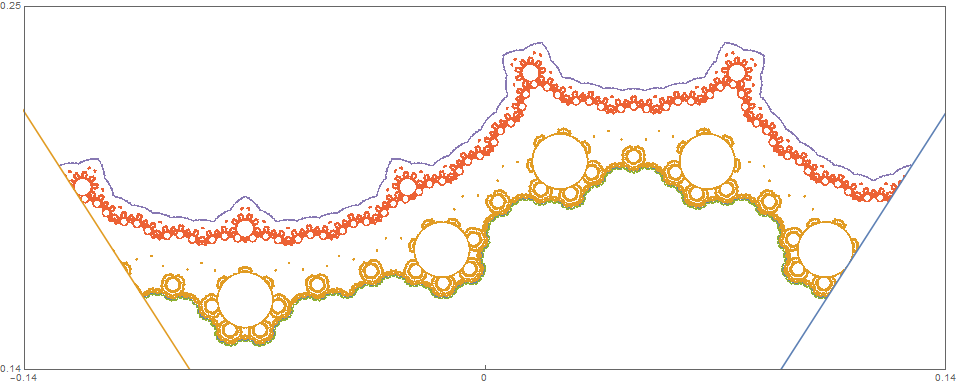}
		\caption{}
		\label{fig:InvariantRange014014014025}
	\end{subfigure}\\
	\begin{subfigure}{.32\textwidth}
		\centering
		\includegraphics[width=1\linewidth]{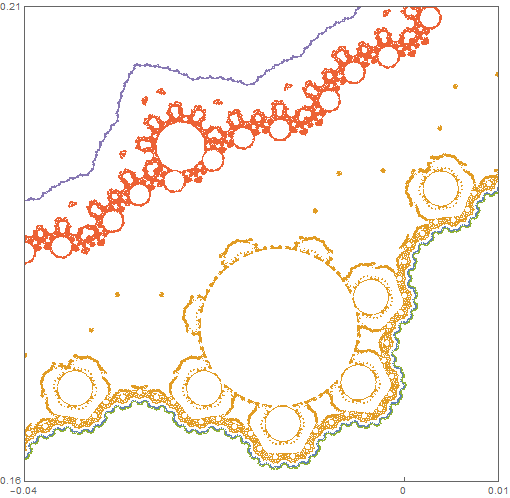}
		\caption{}
		\label{fig:InvariantRange004001016021}
	\end{subfigure}
	\begin{subfigure}{.38\textwidth}
		\centering
		\includegraphics[width=1\linewidth]{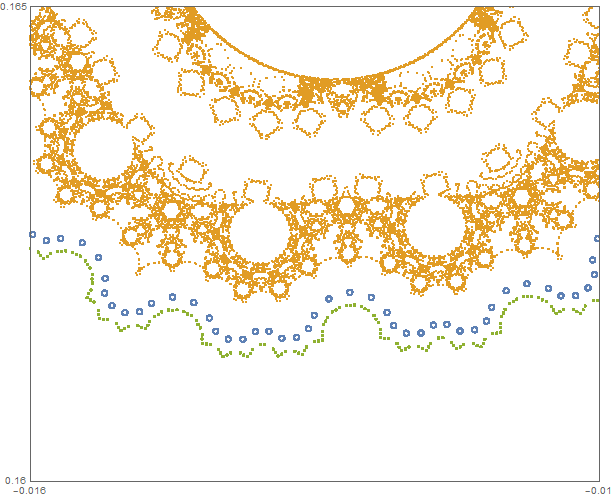}
		\caption{}
		\label{fig:InvariantRange00160010160165}
	\end{subfigure}
	\centering
	\captionsetup{width=1\textwidth}
	\caption{ (A) Orbits of 5 points (ignoring a transient) under $R$, for $(\alpha, \beta, \lambda)= (0.5, 1,\frac{\sqrt{5}-1}{2})$ 
		in the area $[-0.14,0.14]\times[0,0.25]$. (B) Details of (A) in the area  $[-0.04,-0.01]\times[0.16,0.21]$.  (C) Details of (A) in the area $[-0.0016,-0.01]\times[0.16,0.165]$. Observe a complex pattern of periodic islands, the presence of non-trivially embedded IETs as well as orbits with more complex structure.}
	\label{fig:dual}
\end{figure}


For typical choices of parameters $\alpha$, $\beta$ and $\lambda$ it seems that the dynamics of $T$ defined by (\ref{e:pwi4}) (and hence of $R$) is very rich.
Figure \ref{fig:Attractors1} (A) shows typical trajectories (after a transient), for two hundred randomly selected points and $(\alpha, \beta, \lambda)= (0.5, 1,\frac{\sqrt{5}-1}{2})$. Details of some invariant sets are then shown in Figure \ref{fig:Attractors1} (B). These numerical simulations illustrate that (as expected \cite{AG06,AG06a}) the map $T$ has an abundance of periodic islands for typical values of the parameters.

Figure \ref{fig:dual} (A) shows the orbits of 5 points (ignoring a transient) under  $R$, for $(\alpha, \beta, \lambda)= (0.5, 1,\frac{\sqrt{5}-1}{2})$. Details of this are shown in Figure \ref{fig:dual} (B) and (C) in the areas  $[-0.04,-0.01]\times[0.16,0.21]$ and $[-0.0016,-0.01]\times[0.16,0.165]$ respectively.

%
%

These figures show the diverse types of behaviour that can be found in the invariant sets of $R$ (and hence $T$). They show what seem to be non-trivial embedded IETs as well as invariant sets of higher dimension. There are also periodic islands to which the return map is a rotation.


Numerical results show that for some parameters we can observe non smooth invariant curves for the dynamics of the map $R$ as defined in equation \eqref{n1}. These curves appear to have a dynamics similar to that of an interval exchange transformation. These curves can bound invariant regions that exhibit quite complex dynamics. 
We now construct one such region: set $\alpha=0.5$, $\beta=1$,  $\lambda=\frac{\sqrt{5}-1}{2}$ and $\eta=1-\lambda$. Consider the points 
$$
z_0=r_0 e^{i(\pi-\beta)}, \quad z_1=r_1 e^{i(\pi-\beta)},
$$
with $r_0 = 0.470$ and $r_1=0.503$ and denote the orbit closures of these points as $\Xi'$ and $\Xi''$. These are contained in the boxed region in Figure \ref{fig:Attractors1} (B) and are also represented in Figure \ref{fig:invstrip} where it can be seen that both $\Xi'$ and $\Xi''$ appear to be non-trivial continuous embeddings of IETs.
\begin{figure}[t]
	\begin{subfigure}{1\textwidth}
		\centering
		\includegraphics[width=1\linewidth]{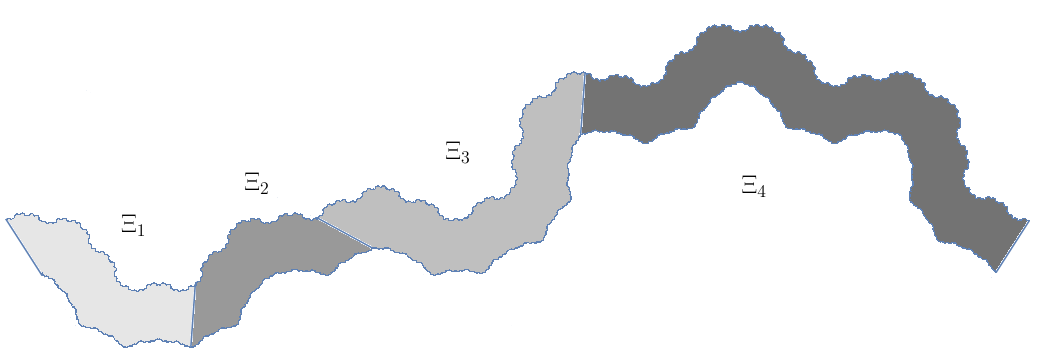}
		\caption{}
		\label{fig:invstrip1}
	\end{subfigure}\\
	\begin{subfigure}{1\textwidth}
		\centering
		\includegraphics[width=1\linewidth]{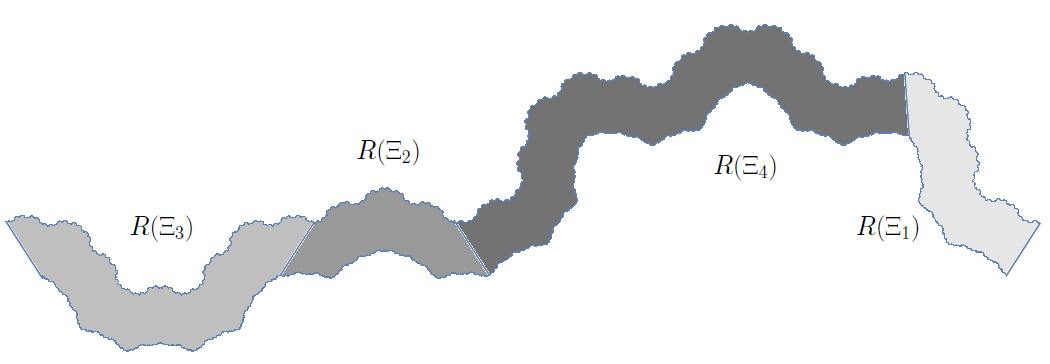}
		\caption{}
		\label{fig:invstrip2}
	\end{subfigure}
	\centering
	\captionsetup{width=1\textwidth}
	\caption{(A) The presumably invariant region $\Xi=\Xi_1 \cup \Xi_2 \cup \Xi_3 \cup \Xi_4$. (B) Image of $\Xi$ by $R$.}
	\label{fig:invstrip}
\end{figure}
Now consider the sets 
$$
Q_L'=\{z \in \C: \arg(z)= \pi-\beta \ \textrm{and} \ r_0 \leq |z| \leq r_1\},
$$
$$
Q_R'=\{z \in \C: \arg(z)= \beta \ \textrm{and} \ r_0 \leq |z| \leq r_1\}.
$$
If $\Xi'$ and $\Xi''$ are invariant curves that are embeddings of IETs, then the set $\partial \Xi=Q_L'\cup Q_R' \cup \Xi' \cup \Xi''$ is a Jordan curve. Denote by $\Xi$ the closure of the region bounded by $\partial \Xi$. Numerical investigations suggest that $\Xi$ is an invariant region for $R$. Let $\Xi_k=Q_k \cap \Xi$, where
\begin{equation*}
\begin{array}{l}
\vspace{0.2cm}
Q_1=\{z\in \C: \Im(e^{-i(\alpha+\beta)}(z+(2\lambda-1)e^{i\alpha}))>0 \},\\
\vspace{0.2cm}
Q_2=\{z\in \C: \Im(e^{-i(\alpha+\beta)}(z+(2\lambda-1)e^{i\alpha}))\leq 0 \ \textrm{and} \ \Im(e^{i(\beta-\alpha)}(z-(1-\lambda)e^{i\alpha}))<0 \},\\
\vspace{0.2cm}
Q_3=\{z\in \C:  \Im(e^{i(\beta-\alpha)}(z-(1-\lambda)e^{i\alpha}))\geq 0 \ \textrm{and} \ \Im(e^{-i(\alpha+\beta)}z)> 0 \},\\
Q_4=\{z\in \C: \Im(z e^{-i(\alpha+\beta)})\leq 0 \}.
\end{array}
\end{equation*}
Using the property of the golden mean $1-\lambda=\lambda^2$ it can be seen that $R(z)=R_j(z)$, for $z \in \Xi_j$ where 
\begin{equation}\label{eq40}
R_j(z)=\left\{\begin{array}{ll} 
ze^{i\vartheta_2}+\lambda^3 , & j=1,\\
ze^{i\vartheta_2}- \lambda^4, & j=2,\\
ze^{i\vartheta_2}-\lambda^2 , & j=3,\\
ze^{i\vartheta_1}+\lambda^3, & j=4.
\end{array}
\right.
\end{equation}
 The subsets $\Xi_j$, $j=1,...,4$ and the action of $R$ in this set are depicted in Figure \ref{fig:invstrip}. Note that that $R$ acts isometrically on each $\Xi_j$, but since these sets are not convex $(\Xi,R)$ is not a 4-PWI, but it is simple to construct a 4-PWI $(\C,S)$ such that $\Xi$ is invariant under $S$ and the restriction of $S$ to $\Xi$ is equal to $R$, by partitioning $\C= \bigcup_{j=0}^4 Q_j$ and setting $S(z)=R_j(z)$, for $z \in Q_j$. One can verify that $S$ satisfies the parametric connecting equation \eqref{ce25}, therefore satisfying a necessary condition for the existence of an IET that can be continuously embedded by $h$ in $(\C,S)$, with $Y=h(I)\subseteq \Xi$ also invariant under $R$.

\subsection{A PWI with an embedded four interval exchange} Finally, we show that the map $R$ in \eqref{eq40} is an example of a $4$-PWI for which numerical evidence suggests the existence of a non-trivial embedded $4$-IET.

Consider the family $\mathcal{F}_4$ of four-interval exchange maps $f_{\mu,\pi}:I\rightarrow I$ given by subdividing the interval into four intervals of lengths $\mu=(\mu_1,\mu_2,\mu_3,\mu_4) \in \R_+^4$ with base permutation $\pi=(2)(143)$.

Note that on the real axis $\Im(z)=0$ is a trivial embedding of the (degenerate) four-interval exchange where $\mu=(\lambda,0,0,1)$. Let
$$
Y=\overline{\{R^n(0.416i)\}_{n\in \N}}.
$$
This defines an invariant set which is portrayed in Figure \ref{fig:invariantcurve} that appears to be an embedding of an IET.
We can partition $Y$ by setting $Y_j=Y \cap \Xi_j$, for $j=1,...4$. The length or Lebesgue one dimensional measure $l_j=Leb(Y_j)$ of each $Y_j$ can be numerically estimated to be 
$$l_1 = 0.1217970148, l_2 = 0.1329352086, l_3 = 0.2008884081, l_4 = 0.3550989199$$ 
Fix $\mu = (l_1,l_2,l_3,l_4)$ and consider the IET $(I,f_{\mu,\pi}) \in \mathcal{F}_4$.
Numerical evidence suggests that there is a continuous embedding of $(I,f_{\mu,\pi})$ into $(\C,S)$, by a map $h:I\rightarrow Y$ with $Y \subseteq \Xi$, such that $h(0)=r_0 e^{i\theta_0}$, with $r_0 = 0.47665$, and $\theta_0= 0.68165 \pi$. Indeed numerical verification shows that $i_k'{R(h(0))}=i_k(f_{\mu,\pi}(0))$ for all $k \leq 10^5$, supporting that $h$ is a symbolic embedding.

We can also verify numerically that the condition in Theorem \ref{t:ergodiccondition} holds for this case. Estimating $\xi_j \simeq \frac{e_j(p(8))}{m_j(p(8))}$ where $\xi_1 \simeq 0.718 + 0.125 i$, $\xi_2\simeq 0.538 - 0.512 i$, $\xi_3\simeq 0.460 - 0.438 i$ and $\xi_4\simeq 0.300 - 0.562 i$. For these estimates we get
$$
\left|\sum_{j=1}^{d}\lambda_j\xi_j \mu_j - h(0)\sum_{j=1}^{d}(1-e^{i\theta_j})\xi_j \mu_j\right| \simeq 6.30 \times 10^{-6},
$$
where $\theta_j=\vartheta_2$, for $j=1,2,3$ and $\theta_4=\vartheta_1$.

\begin{figure}[t]
	\includegraphics[width=0.6\linewidth]{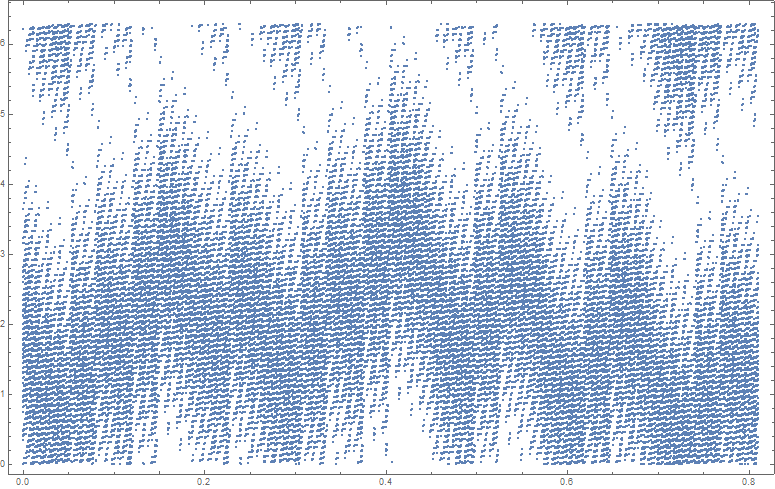}
	\centering
	\captionsetup{width=1\textwidth}
	\caption{First $10^5$ points of the orbit of $(0,0)$, by the tangent exchange map $\Psi$ given by $\mu = (l_1,l_2,l_3,l_4)$, $\pi=(2)(143)$, $\theta_j=\vartheta_2$, for $j=1,2,3$ and $\theta_4=\vartheta_1$. Observe the apparent lack of ergodicity as expected for a non-trivial embedding.}
	\label{fig:Cylinderexchange1}
\end{figure}

Figure \ref{fig:Cylinderexchange1} shows $10^5$ points of the orbit of $(0,0)$, by tangent exchange map $\Psi$ associated to $S$, which is consistent with the orbit being dense but not having nonuniform distribution on the cylinder $I\times S^1$.

\section{Conjectures, questions and conclusions}
\label{sec:discussion}
We have highlighted that embeddings of IETs into PWIs present a number of subtle and mathematically rich problems associated with the regularity or otherwise of these embeddings. Theorem~\ref{thm:trivial2IET} shows that there are no non-trivial continuous embeddings of a minimal $d$-IET into a $d$-PWI, for $d=2$, while Theorem~\ref{t:ergodiccondition} gives a condition for the existence of a piecewise continuous embedding. For $d=4$ there are PWIs that seem to have an abundance of non-trivial embeddings of $d$-IETs. It seems to be much harder to find a $3$-PWI that exhibits non-trivial embeddings of $3$-IETs and to do so requires much parametric fine tuning, a fact that is justified by Theorem \ref{thm:3PWI} which shows that any $3$-PWI has at most one non-trivially continuously embedded minimal $3$-IET with the same underlying permutation.

Our main conjecture is as follows.

\begin{conjecture}
For any $d\geq 3$ there is a minimal $d$-IET $(I,f)$ which admits an embedding a into  $d$-PWI which is continuous but not trivial.
\end{conjecture}

Assuming this is conjecture is valid, a number of interesting lines of enquiry open up:
\begin{itemize}
\item For a given IET $(I,f)$, what is the structure of the PWIs $(X,T)$ that carry continuous embeddings of $(I,f)$, and how can the irregularity of the continuous embeddings be characterised within this class?
\item For a given PWI $(X,T)$, what is the structure of the IETs $(I,f)$ that are embedded within this PWI?
\item So far we have considered continuous embeddings of $d$-IETs into $d$-PWIs but in principle a $d$-IET may be embedded into a $d'$-PWI for some $d'<d$: see for example \cite[Figure 8]{AG06a}. How can we understand these embeddings within PWIs with fewer atoms?
\item What is the structure of parametrizations of $d$-PWIs that embed the same given IET?
\end{itemize}
We suspect that for a continuous non-trivial embedding $h$ of $(I,f)$ into $(X,T)$, typical embeddings have a tangent exchange map that is minimal but not ergodic. A lot of insight has come from examples (see eg \cite{AG06a}) and we suggest that the example introduced in Section~\ref{sec:example} will be useful to explore the above in that it has a larger number of apparent embeddings that limit to the baseline.

We do not consider the case of $(I,f)$ that are not minimal, or that are reducible, but there may be some surprises waiting there as well. The region $\Xi$ discussed in Section~\ref{sec:example} seems to contain periodic islands, embedded IETs and other invariant sets that are neither. It is a challenge to describe these other invariant sets in a coherent way. Regarding the IETs embedded in $\Xi$ we conjecture that all minimal nearby IETs in $\mathcal{F}_4$ are continuously (or at least symbolically) embedded. 



\begin{thebibliography}{Ha}
	
	\bibitem{AKT} Adler, R., Kitchens, B., Tresser, C. (2001). \emph{Dynamics of non-ergodic piecewise affine maps of the
		torus}. Ergodic Theory and Dynamical Systems {\bf 21}, 959-999.
	
	\bibitem{AKMTW} Adler, R., Kitchens, B., Martens, M.,  Tresser, C., Wu, C.W., (2003). \emph{The mathematics of halftoning}. IBM J. Res. Develop., {\bf 47}, 5-15-
	
	\bibitem{AOW85}
	Arnoux, P., Orstein, D., Weiss, B., (1985). \emph{Cutting and stacking, interval exchanges and geometric models}. Israel J. Math. , {\bf 11}, 160--168.
	
	\bibitem{Ash1} Ashwin, P. (1996). \emph{Non-smooth invariant circles in digital overflow oscillations}. 
	Proceedings of the 4th Int. Workshop on Nonlinear Dynamics of Electronic Systems, Sevilla, 417--422.
	
	\bibitem{ACP97}
	Ashwin, P., Chambers, W., Petrov, G., (1997). \emph{Lossless digital filter overflow oscillations; approximation of invariant fractals}. Internat. J. Bifurcation Appl. Eng. 7 , no. 11, 2603--2610.
	
	\bibitem{AFu}
	Ashwin, P., Fu, (2002). \emph{On the Geometry of Orientation-Preserving Planar Piecewise Isometries}. Journal of Nonlinear Science, 12, no. 3, 207--240.
	
	\bibitem{AG06}
	Ashwin, P. ,  Goetz, A.  (2006). \emph{Polygonal invariant curves for a planar piecewise isometry}. Trans. Amer. Math. Soc.  \textbf{358}    no. 1, 373-390.
	
	\bibitem{AG06a}
	Ashwin, P. ,  Goetz, A. (2005). \emph{Invariant curves and explosion of periodic islands in systems of piecewise rotations}. SIAM J. Appl. Dyn. Syst., 4(2), 437-458.
	
	\bibitem{AF} Avila, A., Forni, G. (2007) \emph{Weak mixing for interval exchange transformations and translation flows}, Ann. of Math. (2), 165(2), 637-664.
	
	\bibitem{Bos85}
	Boshernitzan, M., (1985). \emph{A condition for minimal interval exchange maps to be uniquely ergodic}. Duke Math. Journal, {\bf 52}, 723--752.
	
	
	\bibitem{BLP} Bruin, H., Lambert, A., Poggiaspalla, G. and Vaienti, S. (2003). \emph{Numerical analysis for a discontinuous rotation of the torus}. Chaos, 13(2), 558-571.
	
	\bibitem{Buz} Buzzi, J. (2001). \emph{Piecewise isometries have zero topological entropy}. Ergodic Theory Dynam. Systems 21, no. 5, 1371-1377.
	
	
	\bibitem{CFS} Cornfeld, I. P. , Fomin, S. V., Sinai, Ya G. (1982) \emph{Ergodic Theory} Grundlehren der Mathematisches Wissenschaften, 245, Springer-Verlag.
	
	\bibitem{Da} Davies, A. C., (1995). \emph{Nonlinear oscillations and chaos from digital filters overflow}, Phil. Trans. Roy. Soc. A, 353, 85-99.
	
	\bibitem{Dean06}
	Deane, Jonathan H. B. (2006). \emph{Piecewise isometries: applications in engineering}. Meccanica,  \textbf{41},  no. 3, 241–252. 37E99.
	
	\bibitem{Gal85}
	Galperin, G.  (1985). \emph{Two constructive sufficient conditions for aperiodicity of interval exchange}. Theoretical and applied problems of optimization  \textbf{176}    8--16.
	
	
	\bibitem{Goetz1}
	Goetz, A. (2000). \emph{Dynamics of piecewise isometries}. Illinois journal of mathematics, \textbf{44}, 465 -- 478, (2000).
	
	\bibitem{Goetz2}
	Goetz, A., Dynamics of piecewise isometries. Thesis (Ph.D.)–University of Illinois at Chicago. 1996. 
	
	\bibitem{KA} Katok, A. B. (1980) \emph{Interval exchange transformations and some special flows are not mixing}, Israel J. Math. 35, 301-310.
	
	\bibitem{Ke} Keane, M. S., (1975) \emph{Interval exchange transformations}, Math. Z. 141, 25-31.
	
	\bibitem{KWC}
	Kocarev, Lj. , Wu, C. W., Chua, L. O., (1996). \emph{IEEE Trans. Circuits Systems II}, \textbf{43} , no. 3, 234{246}.
	
	\bibitem{LV14}
	Lowenstein, John H., Vivaldi, Franco, \emph{Renormalization of a one-parameter family of piecewise isometries} arXiv:1406.6910 
	
	\bibitem{M} Masur, H., (1982) \emph{Interval exchange transformations and measured foliations}, Ann. of Math. 115, 169-200.
	
	\bibitem{Schwartz} 
	Schwartz, R. E., (2007). \emph{Unbounded Orbits for Outer Billiards}, J. Mod. Dyn. 3, 371-424.
	
	\bibitem{S} Scott, A. J., (2003). \emph{Hamiltonian mappings and circle packing phase spaces: Numerical investigations}. Phys. D, 181, 45-52.
	
	\bibitem{SHM} Scott, A. J., Holmes, C.A., Milburn, G. (2001). \emph{Hamiltonian mappings and circle packing phase spaces}. Phys. D, 155, 34-50.
	
	\bibitem{V1} Veech, W. A., (1982) \emph{Gauss measures for transformations on the space of interval exchange maps}, Ann. of Math. 115, 201-242.
	
	\bibitem{Vi06} Viana, M., (2006) \emph{Ergodic Theory of interval exchange maps}, Revista Matematica Complutense 19, 7-100.
	
	\bibitem{VV} Vivaldi, F., Vladimirov, I., (2003). \emph{Pseudo-randomness of round-off errors in discretized linear maps on the plane}, Internat. J. Bifur. Chaos Appl. Sci. Eng., 13, 3373-3393.
	
\end{thebibliography}
\end{document}